\documentclass[11pt]{amsart}
\usepackage{va, array, multirow, float, bm, color, permute}
\usepackage{colortbl}
\usepackage[pdftex]{hyperref}
\hypersetup{colorlinks,citecolor=blue,linktocpage,hyperindex=true,backref=true}
\usepackage{tikz,amssymb,amscd}
\usetikzlibrary{cd}
\usepackage[arrow, matrix, curve]{xy}
\input xy
\usepackage{todonotes}

\newcommand{\ocA}{\overline{\mathcal A}}
\newcommand{\ogamma}{\overline{\gamma}}
\newcommand{\oJC}{\overline{JC}}
\newcommand{\ocM}{\overline{\mathcal M}}
\newcommand{\op}{\overline{p}}
\newcommand{\opi}{\overline{\pi}}
\newcommand{\tbeta}{\widetilde{\beta}}

\newcommand{\tcG}{\widetilde{\mathcal G}}
\newcommand{\tcM}{\widetilde{\mathcal M}}

\newcommand{\rra}{\rightarrow}
\newcommand{\lra}{\longrightarrow}

\newcommand{\Sym}{\operatorname{Sym}}

\begin{document}
\title{Hodge classes on the moduli space of $W(E_6)$-covers and the geometry of $\mathcal{A}_6$}

\author[V. Alexeev]{Valery Alexeev}
\address{Valery Alexeev: Department of Mathematics, University of Georgia
\hfill \newline\texttt{}
 \indent Athens GA 30602, USA}
\email{{\tt valery@math.uga.edu}}

\author[R. Donagi]{Ron Donagi}
\address{Ron Donagi: Department of Mathematics, University of Pennsylvania
\hfill \newline\texttt{}
\indent 209 South 33rd Street,
Philadelphia, PA 19104-6395, USA} \email{{\tt donagi@math.upenn.edu}}

\author[G. Farkas]{Gavril Farkas}
\address{Gavril Farkas: Institut f\"ur Mathematik, Humboldt-Universit\"at zu Berlin \hfill \newline\texttt{}
\indent Unter den Linden 6,
10099 Berlin, Germany}
\email{{\tt farkas@math.hu-berlin.de}}

\author[E. Izadi]{Elham Izadi}
\address{Elham Izadi: Department of Mathematics, University of California, San Diego \hfill
\indent \newline\texttt{}
\indent La Jolla, CA 92093-0112, USA}
 \email{{\tt eizadi@math.ucsd.edu}}

\author[A. Ortega]{Angela Ortega}
\address{Angela Ortega: Institut f\"ur Mathematik, Humboldt-Universit\"at zu Berlin \hfill \newline\texttt{}
\indent Unter den Linden 6, 10099 Berlin, Germany}
\email{{\tt ortega@math.hu-berlin.de}}

\maketitle

\begin{center}
\emph{To Herb, with friendship and admiration.}
\end{center}

\begin{abstract}
In previous work we showed that the Hurwitz space of $W(E_6)$-covers of the projective line branched over $24$ points dominates via the Prym-Tyurin map  the moduli space $\mathcal{A}_6$ of principally polarized abelian $6$-folds. Here we determine the $25$ Hodge classes on the Hurwitz space of $W(E_6)$-covers corresponding to the $25$ irreducible representations
of the Weyl group $W(E_6)$. This result has direct implications
to the intersection theory of the toroidal compactification $\overline{\cA}_6$. In the final part of the paper, we present an alternative, elementary proof of our uniformization result on $\mathcal{A}_6$ via Prym-Tyurin varieties of type $W(E_6)$.
\end{abstract}

\newcommand\gen{_{\rm gen}}
\newcommand\syz{_{\rm syz}}
\newcommand\azy{_{\rm azy}}
\newcommand\perf{^{\rm perf}}
\newcommand\sat{^{\rm sat}}
\newcommand\oasp{{\oA_6}}
\newcommand\oagp{{\oA_g}}
\newcommand\ocH{{\overline{\mathcal H}}}
\newcommand\och{\ocH}
\newcommand\torrk{{\rm tor.rk}}
\newcommand\fn{\mathfrak{n}}
\newcommand\lambdam{\lambda^{(-5)}}
\newcommand\lambdap{\lambda^{(+1)}}
\newcommand\chim{\chi^{(-5)}}
\newcommand\chip{\chi^{(+1)}}
\newcommand\hur{\rm Hur}
\newcommand\Hur{\hur}
\newcommand\thur{\widetilde{\hur}}
\newcommand\ohur{{\overline{\hur}}}
\newcommand\conj{c}
\newcommand\slope{{\rm slope}}
\newcommand\Res{\operatorname{Res}}
\newcommand\WE{W(E_6)}
\newcommand\EE{E_6}
\newcommand\junk{(\star)}
\newcommand\modjunk{{\mod\junk}}
\newcommand\tov[1]{$\overline{\text{#1}}$}
\newcommand\ov[1]{\overline{\text{#1}}}
\newcommand\trace{\operatorname{trace}}

\newcommand\hh{\ocH}
\newcommand\mm{\oM}
\newcommand{\oq}{\overline{q}}
\newcommand{\of}{\overline{f}}
\newcommand\lcm{\mathrm{lcm}}

\numberwithin{equation}{section}
\def\theequation{\arabic{section}\Alph{equation}}

\setcounter{tocdepth}{1}

\section{Introduction}
\label{sec:intro}

It is well known that the moduli space $\cA_g$ of principally polarized
abelian varieties of dimension $g\leq 5$ can be uniformized via Prym varieties associated to unramified double covers of curves. This amounts to the fact that the Prym map $P\colon\cR_{g+1}\rightarrow \cA_g$ is dominant in this range. This explicit parametrization of the moduli space has important applications, for instance it implies that $\cA_g$ is unirational for $g\leq 5$, see \cite{donagi1984the-unirationality,
mori1982uniruledness-moduli, C, verra1984short-proof}. Note also that $\cA_g$ is a variety of general type for $g\ge7$, see
\cite{mumford1983on-the-kodaira-dimension, tai1982on-the-kodaira}. Using advances in automorphic forms, it has been recently proven \cite{salvati} that the Kodaira dimension of $\mathcal{A}_6$ is non-negative.

\vskip 4pt

There is a well documented history going back at least to \cite{D3} showing the importance of the symmetries of the $27$ lines on a cubic surface in the study of the Galois group of the Prym map $P\colon \mathcal{R}_6\rightarrow \mathcal{A}_5$. Conversely, Clemens and Groffiths \cite{CG} famously associated to a smooth cubic threefold its intermediate Jacobian in order to study rationality questions. For recent developments in moduli theory or in hyperk\"ahler geometry related to this circle of ideas we refer to \cite{CMGHL, LSV, V2}.

\vskip 3pt

In our previous paper \cite{alexeev15the-uniformization} we found an explicit parametrization of
$\cA_6$ by means of one-dimensional objects. Recalling that $\WE$ is the group of symmetries of the $27$ lines on a smooth cubic surface, we  proved that the general ppav $[A,\Theta]\in \cA_6$ can be represented as the Prym-Tyurin
variety of exponent $6$ associated to an $\WE$-cover $\pi\colon C\rightarrow \mathbb P^1$ branched over $24$ points.  Precisely, let $\hur$ denote the Hurwitz space of covers $[\pi\colon C\rightarrow \bP^1, p_1+\cdots+p_{24}]$ having monodromy group $W(E_6)\subseteq S_{27}$ and branched over the points $p_1, \ldots, p_{24}\in \mathbb P^1$ such that the local monodromy of $\pi$ at $p_i$ is given by a reflection in a root of $\EE$. For each such cover $\pi\colon C\rightarrow \mathbb P^1$ we can identify the points in a general fiber  with the lines on a smooth cubic surface. The curve $C$ has genus $46$ and is equipped with an \emph{incidence correspondence} $D\subseteq C\times C$ first considered by Kanev \cite{kanev1989spectral-curves}. The correspondence $D$ gives rise to an endomorphism $D\colon JC\rightarrow JC$ and to a
\emph{Prym-Tyurin-Kanev map}

$$PT \colon \mathrm{Hur} \rightarrow  \cA_6,  \ \ \ [\pi \colon C\rightarrow \mathbb P^1] \mapsto PT(C,D):=\mbox{Im}(D-1)\subseteq JC.$$

Since $(D-1)(D+5)=0$, one has $PT(C,D)=\mbox{Ker}(D+5)^{0}$. Our main result from \cite{alexeev15the-uniformization} is that the map $PT$ is generically finite, in particular dominant. This parametrization opens the way to a study of $\cA_6$ via the theory of curves and their correspondences. The main goal of this paper is to understand the intersection theory associated to this uniformization of $\cA_6$, in particular to determine the $25$ Hodge classes associated to the irreducible representations of the group $\WE$.

\vskip 3pt

The moduli space $\cA_g$ has a partial compactification $\cA_g^*$ obtained
by adding rank~1 degenerations and contained in the toroidal
compactification $\overline{\cA}_g = {\overline{\cA}_g^{\rm perf}}$ for the fan of
perfect forms, with the complement $\overline{\cA}_g\setminus \cA_g^*$ having
codimension 2.  The Hurwitz space $\hur$ has a modular compactification $\ohur$ by means of $\WE$-admissible covers.
The Prym-Tyurin map $PT$ extends to a rational map $$PT\colon \ohur\dashrightarrow \overline{\cA}_6$$ with indeterminacy locus of codimension at least $2$.
Although the Hurwitz space $\ohur$ has an intricate divisor theory, with boundary divisors associated to complicated discrete data, it is one of the important results of \cite{alexeev15the-uniformization} that only three explicitly described boundary divisors $D_0, D_{\mathrm{azy}}, D_{\mathrm{syz}}$ of $\ohur$ are not contracted under the map $PT$.
Here $D_{\mathrm{azy}}$ and $D_{\mathrm{syz}}$ denote the boundary divisors of \emph{azygetic} (respectively \emph{syzygetic}) $\WE$-admissible covers,  having as general element a cover $$[\pi\colon C=C_1\cup C_2\rightarrow R_1\cup_q R_2, \ p_1+\cdots+p_{24}],$$ with $\pi^{-1}(R_i)=C_i$
for $i=1,2$, where $R_1$ and $R_2$ are smooth rational curves meeting at the point $q$, precisely two branch points, say $p_{23}$ and $p_{24}$, lie on $R_2$ and the \emph{distinct} roots $r_{23}, r_{24}\in E_6$ determining the local monodromy at the corresponding points satisfy $r_{23}\cdot r_{24}\neq 0$ (respectively $r_{23}\cdot r_{24}=0$). The divisor $D_0$ corresponds to the situation when the roots $r_{23}$ and $r_{24}$ are equal.
In order to study $\overline{\cA}_6$, it suffices therefore to restrict our attention to the partial compactification of the Hurwitz space

$$\widetilde{\hur}:=\hur\cup D_0 \cup D_{\mathrm{azy}}\cup D_{\mathrm{syz}}\subseteq \ohur.$$

The divisor $D_0$ is mapped onto the the boundary divisor $D_6:=\overline{\mathcal{A}}_6\setminus \mathcal{A}_6$, whereas $D_{\mathrm{syz}}$ and $D_{\mathrm{azy}}$ are mapped onto divisors of $\overline{\mathcal{A}}_6$ not contained in the boundary.

\vskip 4pt

The  Kanev correspondence $D\subseteq C\times C$ can be extended for any point $[\pi\colon C\rightarrow R, p_1+\cdots+p_{24}]\in \ohur$. In particular, it induces a decomposition
\begin{equation}\label{decomp}
H^0(C, \omega_C)=H^0(C, \omega_C)^{(+1)}\oplus H^0(C, \omega_C)^{(-5)}
\end{equation}
into $(+1)$ and $(-5)$ eigenspaces with respect to $D$ and having dimensions $40$ and $6$ respectively. We denote by $\lambda, \lambda^{(+1)}$ and
$\lambda^{(-5)}$ the Hodge eigenbundles on $\ohur$ globalizing the decomposition (\ref{decomp}) over the entire moduli space. If $\lambda_1\in CH^1(\overline{\mathcal{A}}_6)$ denotes the Hodge class, since $PT^*(\lambda_1)=\lambda^{(-5)}$ and $K_{\overline{\mathcal{A}}_6}=7\lambda_1-[D_6]$, where $D_6$ is the boundary divisor of $\overline{\cA}_6$ of rank $1$  degenerations,  determining the class $\lambda^{(-5)}$ is essential to any further investigation of the birational geometry of $\overline{\mathcal{A}}_6$. One of the main results of this paper is that $\lambda^{(-5)}$ has a remarkably simple expression:

\begin{theorem}\label{mainthm_1}
The class of the $(-5)$-Hodge eigenbundle  on $\widetilde{\Hur}$ is given by the following formula:
$$6\lambda^{(-5)}=\lambda-\frac{1}{2}[D_{\mathrm{syz}}].$$
\end{theorem}

Since it has been shown in \cite[Theorem 6.17]{alexeev15the-uniformization} that the Hodge class $\lambda$ on $\ohur$ can be expressed in terms of boundary divisors, Theorem \ref{mainthm_1} can be rewritten using only $D_0, D_{\mathrm{syz}}$ and $D_{\mathrm{azy}}$ and one has the following identity on $\widetilde{\hur}$:
\begin{equation}\label{eq:lam5}
\lambdam = \frac{11}{92}[D_0] - \frac1{46}[D\syz] + \frac7{276}[D\azy].
\end{equation}

\vskip 4pt

Our approach to proving Theorem \ref{mainthm_1} is representation-theoretic: The Weyl group $\WE$ has $25$ irreducible representations $\rho_1, \ldots, \rho_{25}$. Each of these determines a variant $\bE_i$ of the Hodge vector bundle over $\ohur$. At a point given by the $27$-sheeted cover $[\pi\colon C \to R,  p_1+\cdots+p_{24}]\in \ohur$ with Galois closure $\widetilde{\pi}\colon {\widetilde{C}} \to R$, the fiber of $\bE_i$ is defined to be $\Hom_{\WE}\bigl(\rho_i, H^0(\widetilde{C}, \omega_{\widetilde{C}})\bigr)$.
The Hodge classes in question are defined as
$\lambda_i :=c_1(\bE_i)$, for $i=1, \ldots, 25$.
The Prym-Hodge bundles $\lambda^{(+1)}$ and $\lambda^{(-5)}$ are two special cases of this construction, obtained from the two non trivial representations of $\WE$ that occur in the standard $27$-dimensional permutation representation of $\WE$. This gives the relation
$\lambda^{(+1)} + \lambda^{(-5)} = \lambda$.
Every representation $\rho_i$ occurs in some permutation representation and every permutation representation gives rise to an associated cover, and the Hodge bundle arising from such a cover decomposes into contributions coming from the various classes $\lambda_i$.
We calculate the Hodge bundles corresponding to a sufficiently large collection of such permutation representations, and use  representation theory to extract from these formulas  the formulas for the Hodge bundles $\lambda_i$ corresponding to all 25 irreducible representations of $\WE$. The permutation representations we use are quotients of the Galois cover $\widetilde{C}$ by cyclic subgroups $W_{\alpha}$ generated by representatives of the $25$ conjugacy classes in $\WE$. The list for the expression of the Hodge classes $\lambda_1, \ldots, \lambda_{25}\in CH^1(\widetilde{\hur})$ can be found in the statement of Theorem \ref{thm:25-curves}.

\medskip

Another important result of this paper concerns the class of the \emph{Weyl-Petri divisor} on $\ohur$. For a smooth $\WE$-cover $\pi\colon C\to\bP^1$ the Weyl-Petri map is the
multiplication map
\begin{displaymath}
  \mu(L)\colon H^0(C, L) \otimes H^0(C, \omega_C\otimes L^\vee) \to H^0(C, \omega_C),
\end{displaymath}
where $L = \pi^*\cO_{\bP^1}(1)\in W^1_{27}(C)$. By
\cite[Theorem 9.2]{alexeev15the-uniformization}, the map $\mu(L)$ is
injective for a general point of $\hur$. Furthermore, it factors through the $(+1)$-eigenspace, that is, one has a map
\begin{equation}\label{eq:weyl-petri}
  \mu(L)\colon H^0(C, L) \otimes H^0(C, \omega_C\otimes L^\vee) \to H^0(C, \omega_C)^{(+1)}.
\end{equation}
Therefore, since its source and target have the same rank, its degeneracy locus is a divisor $\mathfrak{N}$
on the space of admissible $\WE$-covers (see Section \ref{sec:ram} for a more precise definition and a discussion of what happens when $h^0(C, L)$ jumps). Our next result determines the class of $\mathfrak{N}$ on $\widetilde{\hur}$:

\begin{theorem}\label{thm:weyl-petri}
The class of the Weyl-Petri divisor on $\widetilde{\hur}$ is given by the following formula:
\begin{equation}
  \label{eq:n-on-hur}
  [\mathfrak{N}] = \frac{59}{42}\lambda -\frac{12}{7}[D_0] - \frac{29}{84}[D\syz].
\end{equation}
\end{theorem}

The proof of Theorem \ref{thm:weyl-petri} involves passing to an alternative partial compactification $\widetilde{\mathcal{G}}_{E_6}$ of $\hur$
over which the multiplication map (\ref{eq:weyl-petri}) can be defined globally, then reinterpreting the obtained result on $\widetilde{\hur}$.

In \cite[Theorem  0.4]{alexeev15the-uniformization} we showed
that if  $[\pi\colon C\to\bP^1]\in\hur$ does not lie in the
Weyl-Petri divisor $\mathfrak{N}$ then it lies in the ramification locus of the
Prym-Tyurin map $PT\colon \hur\to \mathcal{A}_6$ if and only if the
\emph{Prym-Tyurin canonical curve} $\varphi_{(-5)}(C)\subseteq \bP H^0(C, \omega_C)^{(-5)}\cong \bP^5$ induced by the sublinear system $\bigl|H^0(C,\omega_C)^{(-5)}\bigr|$
lies on a quadric, that is, the multiplication map
$$\mbox{Sym}^2 H^0(C, \omega_C)^{(-5)}\rightarrow H^0(C, \omega_C^{\otimes 2})$$
in not injective. We clarify the set-theoretic description of the ramification divisor of $PT$:

\begin{theorem}\label{thm:ram}
The ramification divisor of the Prym-Tyurin map $PT\colon \hur\rightarrow \mathcal{A}_6$ is contained in the union of the Weyl-Petri divisor $\fN$ and the effective divisor $\fM$  parametrising $W(E_6)$-covers $[\pi\colon C\rightarrow \bP^1]$ such that $h^0\bigl(C, \pi^*(\cO_{\mathbb P^1}(1)\bigr)\geq 3$.
\end{theorem}

The fact that the condition $h^0(C,L)\geq 3$ for $L=\pi^*(\mathcal{O}_{\bP^1}(1))$ defines a divisor $\fM$ on $\hur$ comes to us as a surprise, for general Brill-Noether theory would predict that such curves depend on considerably fewer moduli. For the precise definition of the divisor $\fM$, we refer to (\ref{fm}).

\vskip 3pt

By analysing directly the differential of the map $PT$ at a general point of the boundary divisor $D_0$, we give a second, more
elementary proof of the main result from \cite{alexeev15the-uniformization}.

\begin{theorem}\label{thm:second-proof}
The Prym-Tyurin map $PT$ is generically unramified along the boundary divisor $D_0$ of $\ohur$. It follows once more that $PT\colon \ohur\dashrightarrow \overline{\mathcal{A}}_6$ is generically finite.
\end{theorem}

We recall that the original proof of the dominance of $PT$ amounted to the \emph{tropicalization} of the Prym-Tyurin map. Precisely, we studied the principal term of the Prym-Tyurin map by expanding the monomial coordinates near the neighborhood of a maximally degenerate cover and then used the theory of degenerations of Prym-Tyurin varieties. This time, the proof, which we complete in Section \ref{sect:dominance} is more direct. The element of $D_0$  for which Theorem \ref{thm:second-proof} is verified is obtained by choosing judiciously $12$ points $q_1, \ldots, q_{12}\in \mathbb P^1$ together with roots $r_1, \ldots, r_{12}\in \EE$, determining a degree $27$ stable map $\pi\colon C\rightarrow \mathbb P^1$, where $C$ is the curve obtained from the disjoint union of $27$ copies of $\mathbb P^1$ labeled by the $27$ lines on a smooth cubic surface and then gluing over each point $q_i$ the components labeled by the double-six corresponding to the root $r_i$.
The verification that the $\WE$-admissible cover associated to $\pi$ verifies all required properties is completed in Theorem \ref{thm:glued-curve-comps}.

\vskip 5pt

{\small{\noindent {{\bf{Acknowledgments:}}} Alexeev was partially supported by the NSF grant DMS-1902157. Donagi was partially supported  by the NSF grant DMS-2001673 and by the Simons Foundation Collaboration grant No 390287 \emph{Homological Mirror Symmetry}.
Farkas was partially supported by the DFG Grant \emph{Syzygien und Moduli} and by the ERC Advanced Grant SYZYGY. This project has received funding from the European Research Council (ERC) under the European Union Horizon 2020 research and innovation program (grant agreement No. 834172).}
This material is partially based upon work partially supported by the National Science Foundation under Grant No. 1440140, while Alexeev, Farkas, Izadi and Ortega were in residence at the Mathematical Sciences Research Institute in Berkeley, California, during the Spring of 2019.}

\section{The Weyl group of $E_6$ and the uniformization of $\cA_6$.}
\label{sec:previous-work}

We give a summary of some group theoretic facts and the results established in
\cite{alexeev15the-uniformization} that are used in this paper.

\subsection{The group $\bm{\WE}$ and its representations}
\label{sec:Hurwitz-space}

Let $\WE$ be the Weyl group of the root lattice $E_6$.  It is the subgroup of
the orthogonal group $\mathbb O(E_6)$ generated by reflections $r_{\alpha}\colon x\mapsto x+(x,\alpha)\alpha$ in a root
$\alpha$ of $E_6$.  One has
$|\WE|=51840$ and $\WE$ has $25$ irreducible representations. The dimensions of these representations are 1, 1, 6, 6, 10, 15, 15, 15, 15, 20, 20, 20, 24, 24, 30, 30, 60, 60, 64, 64, 80, 81, 81, 90. In order to refer to the characters and  conjugacy
classes of $\WE$ we use the notation from the character table from the Atlas
\cite[p.27]{conway1985atlas-finite} for the group $U_4(2).2=\WE$. It
is obtained from the character table of $U_4(2)$ by the splitting and
fusion rules. It can be reproduced in GAP \cite{GAP4} by
using the command {\tt Display(CharacterTable("W(E6)"))}.

In addition to the numbers 1,  \ldots, 25 for the characters of $W(E_6)$, we use
convenient names, as in Table~\ref{tab:E6-chartable}. They start
with the dimension of the representation and add attributes \emph{a, b},
and so on, if there are several irreducible representations of the same dimension. We also group
characters in pairs $\chi$ and $\overline{\chi} = \chi \otimes \ov{1}$
whenever these are different.  Here, ${\ov{1}}$ is the 1-dimensional
character of $\WE$ sending an element $u\in \WE$ to $(-1)^n$ if $u$ is a product
of $n$ reflections.

\vskip 3pt

\begin{notation}\label{notation:cubicsurfaces}
We use repeatedly the geometric realization $E_6\cong K_S^{\perp}\subseteq \mbox{Pic}(S)$, where $S$ is a smooth cubic surface. We use the classical notation
$a_1, \ldots, a_6$, $b_1, \ldots, b_6$ and $c_{ij}$, for $1\leq i<j\leq 6$ for the $27$ lines on $S$. A system of fundamental roots
of $E_6$ is then given by $\omega_i:=a_i-a_{i+1}$ for $i=1, \ldots, 5$ and $\omega_6:=h-a_1-a_2-a_3$, where $h := - K_S$ is the hyperplane class.
\end{notation}

\begin{notation}\label{notation:conjclass}
We record three important conjugacy classes in the Weyl group $\WE$, namely the class 2c containing
reflections~$w\in W(E_6)$, the class 2b containing products $w_1\cdot w_2$ of two commuting
(syzygetic) reflections $w_1, w_2\in W(E_6)$, and 3b containing products $w_1\cdot w_2$ of two non-commuting (azygetic) reflections.
\end{notation}

The character table of $\WE$, playing a significant role in several of our calculations is reproduced in the appendix of this paper as Table \ref{tab:E6-chartable}.  We fix representatives $w_i$ of the 25 conjugacy classes in
  $\WE$, labeled so that: $w_{1a}=1$, $w_{2c}$ is a reflection, that is, a representative  of the class 2c in the notation of the character table of $\WE$, then $w_{2b}$ is the product of two syzygetic reflections and so on.

\begin{notation}
  For an element $u\in \WE$, we denote by $Z_u$ its centralizer in
  $\WE$ and by $c_u$ its conjugacy class in $\WE$.
\end{notation}

Assume now that $G$ is a subgroup of $\WE$ of index $d$ and let $u\in \WE$ be a fixed element. The assignment $xG \mapsto uxG$ induces a bijection on the sets $\WE/G$ of left cosets and can thus be regarded as a permutation from $S_d$.  We shall need the following simple group-theoretic fact.

\begin{lemma}\label{lem:cycle-type}
Let $u\in \WE$ be an element of prime order $p$. Then its cycle type
in $S_d$ is $p^a1^b$, where
\begin{equation}\label{eq:cycle-type}
 b = \frac{ |G \cap c_u| \cdot \ |Z_u| } { |G| }, \qquad
 a =  \frac{d - b}{p}.
\end{equation}
\end{lemma}
\begin{proof}
We consider the bijection $\WE/G\rightarrow \WE/G$ on the set of $G$-cosets induced by multiplication with $u$.  Since $u\in \WE$ has prime order $p$, there
are only two possibilities for a coset $xG$. It is either fixed,  or its
orbit consists of exactly $p$ cosets. We first count the number of elements $x\in \WE$ such that $uxG=xG$. In this case
$x^{-1}u x=:u' \in G\cap c_u$. We consider the surjective map $\chi_u\colon \WE\rightarrow c_u$ given by $\chi_u(x):=x^{-1}ux$. Each fibre of $\chi_u$ consists of $|Z_u|$ elements, thus the number of elements $x$ with $uxG=xG\in \WE/G$ equals $|G\cap c_u|\cdot |Z_u|$. In order to obtain the number of $u$-fixed $G$-cosets we have to divide this number by $|G|$, which gives the stated formula for
$b$. Then $a$ is computed from the equality $pa+b=d$.
\end{proof}

The quantities $a$ and $b$ computed in Lemma \ref{lem:cycle-type} clearly depend only on the conjugacy class $c_u$ of $u$. In particular, when the subgroup $G$ is fixed, we obtain a vector of positive integers
\begin{equation}\label{vector3}
\bigl(a_{2c}, b_{2c}, a_{2b}, b_{2b}, a_{3b}, b_{3b}\bigr).
\end{equation}
Since the order of the representatives $w_{2c}$ and $w_{2b}$ is equal to $2$, whereas $\mbox{ord}(w_{3b})=3$, one has
$$2a_{2c}+b_{2c}=2a_{2b}+b_{2b}=3a_{3b}+b_{2b}=[\WE:G]=d.$$

\subsection{Maximal subgroups of $\WE$}
\label{sec:maxsugroups}

Up to conjugation, the group $\WE$ has five maximal subgroups, see \cite[Theorem 9.2.2]{Dolgachev2012}.

\noindent $\bullet$ A subgroup $G_{27}\subseteq \WE$ of index $27$, which can be viewed as the stabilizer of a line of the cubic surface $S$ under the identification $E_6\cong K_S^{\perp}$. One has $G_{27}\cong W(D_5)$. In this paper we constantly make the choice $G_{27}:=\mbox{Stab}_{\WE}(a_6)=\langle \omega_1, \omega_2, \omega_3, \omega_4, \omega_6\rangle$.

\noindent $\bullet$ A subgroup $G_{36}\subseteq \WE$ of index $36$, viewed as the stabilizer of a double six on $S$.

\noindent $\bullet$  A subgroup $G_{45}\subseteq \WE$ of index $45$, regarded as the stabilizer of a tritangent plane of $S$. Note that $G_{45}\cong W(F_4)$.

\noindent $\bullet$ Two subgroups $G_{40}$ and $G_{40}'$ of index $40$.

For instance, for the subgroup $G_{27}$ the vector described in (\ref{vector3}) is equal to $$\bigl(a_{2c}, b_{2c}, a_{2b}, b_{2b}, a_{3b}, b_{3b}\bigr)=(6,15, 10, 7, 6, 9).$$

\subsection{Three versions of compactified Hurwitz spaces of $\WE$-covers}
\label{sec:H-Hur_GE6}

We denote by  $\cH$ the Hurwitz space of smooth $\WE$-covers $[\pi\colon C\rightarrow \bP^1, p_1, \ldots, p_{24}]$ together with a labeling of its branch points.  The map $\pi$ is of degree $27$. The global
monodromy of $\pi$ equals $\WE$ and the local monodromy around each branch point $p_i\in \bP^1$ is a reflection in a root of $E_6$, that is, an element in the conjugacy class 2c in the notation of the character table of $W(E_6)$. The curve $C$ is smooth of genus 46 and the cover $\pi\colon C\rightarrow \bP^1$ is not Galois.

\vskip 3pt

Let $\hh$ be the compactification of $\mathcal{H}$ by admissible $W(E_6)$-covers. This can be regarded as the stack of \emph{balanced
 twisted stable} maps into the classifying stack $\mathcal{B} W(E_6)$ of $W(E_6)$, that is,
$$\hh:=\overline{\mathcal{M}}_{0,24}\Bigl(\mathcal{B} W(E_6)\Bigr).$$

The map
$\mathfrak{b} \colon\hh\to \overline{\mathcal{M}}_{0,24}$ forgetting the monodromy data is finite, so
$\dim(\hh)=21$. The symmetric group $S_{24}$ acts on both $\overline{\mathcal{M}}_{0,24}$ and $\hh$ by permuting the marked (respectively branch) points, and we denote the corresponding quotients by
$$\ohur:=\hh/S_{24}\ \ \mbox{ and }\ \ \widetilde{\cM}_{0,24}:=\overline{\cM}_{0,24}/S_{24}.$$
Let $q\colon \hh\rightarrow \ohur$ denote the quotient map. The space $\ohur$ is the main object of study both in \cite{alexeev15the-uniformization} and in the present paper, on which most of the intersection-theoretic formulas  are written.

We have regular maps
$$\mathfrak{br}\colon \ohur\rightarrow \widetilde{\cM}_{0,24} \ \ \mbox{ and } \ \ \widetilde{\varphi}\colon \ohur\rightarrow \overline{\mathcal{M}}_{46}$$
associating to an admissible cover $[\pi\colon C\rightarrow R, p_1+\cdots+p_{24}]\in \ohur$ the branch locus
$[R, p_1+\cdots+p_{24}]$ and the stable model of its source curve $C$ respectively.

\vskip 4pt

The third version of a compactified space of $\WE$-covers is the one that admits a universal $\WE$-line bundle of degree $27$, which is something both $\hh$ and $\ohur$ lack. Following Section 9 of \cite{alexeev15the-uniformization} we denote by $\widetilde{\cG}_{\EE}$ the (normalization of the) moduli space parametrizing finite maps $[\pi : C \rightarrow R]$ with monodromy $\WE$, where $C$ is an irreducible stable curve
of genus $46$ and $R$ is a smooth rational curve. For such a map, $L := \pi^* \cO_R (1)$ is a base point free line bundle of degree $27$ on $C$ with $h^0(C, L)\ge2$.
The  spaces $\ohur$ and $\widetilde{\cG}_{\EE}$ share the open
subspace $\hur$ on which the source curve $C$ is smooth.
We denote
by $$\tilde{f}\colon \widetilde{\mathcal{C}}_{E_6}\rightarrow \widetilde{\mathcal{G}}_{E_6}$$
the universal genus $46$ curve. The fibres of $\tilde{f}$ are irreducible curves of genus $46$.

\vskip 4pt

Following \cite[9.5]{alexeev15the-uniformization}, we denote by
$\widetilde{\beta}\colon \ohur\dashrightarrow \widetilde{\mathcal{G}}_{E_6}$ the birational map assigning to a point $[\pi \colon C\rightarrow \mathbb P^1, p_1+\cdots+p_{24}]\in \hur$ the map $[\pi : C \rra R]\in \widetilde{\mathcal{G}}_{E_6}$. Since $\ohur$ is normal, $\widetilde{\beta}$ can be extended to a regular map outside a subvariety of codimension at least $2$ in $\ohur$.

\subsection{The dominance of the Prym-Tyurin map}
\label{sec:PT-dominance}

A fiber of the cover $\pi\colon C\to\bP^1$ corresponding to an element of $\cH$ has the combinatorial structure of the $27$ lines on
a cubic surface, and the $\WE$-action on each of its fibres preserves the incidence
relation. The correspondence sending a line $\ell$ to the 10 lines
incident to it can be thus regarded as a correspondence on $C$ and it induces an endomorphism $D$ on the Jacobian
$JC:=\Pic^0(C)$, satisfying the quadratic relation $(D-1)(D+5)=0$. By Kanev
\cite{kanev1987principal-polarizations, kanev1989spectral-curves} the
$(-5)$-eigenspace of this endomorphism $$PT(C,D):=\mbox{Ker}(D+5)^0 = \mbox{Im}(D-1)\subseteq JC$$ is a
principally polarized abelian variety of dimension 6 and exponent $6$, which we call
the \emph{Prym-Tyurin variety} of the pair $[C, D]$. This assignment defines the map $PT_\cH\colon\cH\to \mathcal{A}_6$ which factors through
the Prym-Tyurin map $PT\colon\hur\to \mathcal{A}_6$.  By
\cite[Theorem 0.1]{alexeev15the-uniformization} these maps are dominant
and generically finite.

\subsection{Boundary divisors on the Hurwitz space}
\label{sec:Hur-boundary}

The boundary divisors on the moduli space
$\overline{\mathcal{M}}_{0,24}$ of stable $24$-pointed rational curves  are of the
form $\Delta_{0:I}$, with $I\subseteq \{1, \ldots, 24\}$ being a subset such that $|I|\ge2$ and  $|I^c|\ge2$. A general point of $\Delta_{0: I}$  corresponds to a $24$-pointed stable rational
curve $[R, p_1, \ldots, p_{24}]$ consisting of two smooth components $R_1$ and $R_2$ meeting at a single point, with the marked points $\{p_i\}_{i\in I}$  (respectively $\{p_j\}_{j\in I^c}$)
lying on $R_1$ (respectively on $R_2$). For $i=2, \ldots, 12$, we have the $S_{24}$-invariant boundary divisor
$$B_i:=\sum_{|I|=i} \Delta_{0:I}.$$

The boundary divisors of $\och$ correspond to the components of the pull-back $\mathfrak{b}^*(B_i)$ under the map
\begin{equation}\label{branchmap}
\mathfrak{b}:\och\rightarrow \overline{\mathcal{M}}_{0,24}.
\end{equation}

In order to keep track of these divisors, we need further combinatorial data. In addition to the partition
$I\sqcup I^c = \{1,\dotsc, 24\}$, we also have the data of reflections
$\{w_i\}_{i\in I}$ and $\{w_j\}_{j\in I^c}$ in $\WE$ such that
$\prod_{i\in I} w_i = u$, $\prod_{j\in I^c} w_j = u\inv$. The products
are taken in order, and the sequence $w_1, \dotsc, w_{24}$ is defined
up to conjugation by the same element $g\in \WE$.

Let $\mu := (\mu_1, \dotsc, \mu_\ell)$ be the cycle type of the element $u\in \WE$ considered as a permutation in $S_{27}$. Set
\begin{equation}\label{notm}
\frac1{\mu} := \frac1{\mu_1} + \dotsb +\frac1{\mu_\ell} \ \mbox{ and } \ \lcm(\mu) :=
\lcm(\mu_1,\dotsc, \mu_\ell).
\end{equation}

We denote by $\mathcal{P}_i$ the set of partitions $\mu$ of $27$ appearing as products of $i$ reflections in $\WE$ . The possibilities for $\mu\in \mathcal{P}_i$ are listed in \cite[Table 1]{alexeev15the-uniformization}. For $\mu\in \mathcal{P}_i$, let $E_{i:\mu}$ denote the sum of all the divisors of $\och$ whose general point corresponds to an $\WE$-cover
$$t:=\bigl[\pi\colon C\rightarrow R, \ p_1, \ldots, p_{24}\bigr]\in \och,$$
where $[R=R_1\cup_q R_2, p_1, \ldots, p_{24}]\in B_{i}\subseteq \overline{\mathcal{M}}_{0,24}$ is a pointed union of two smooth rational curves $R_1$ and $R_2$ meeting at the point $q$. Over $q\in R_{\mathrm{sing}}$, the map $\pi$ is ramified according to $u$,
that is, the points in $\pi^{-1}(q)$ correspond to cycles in the
permutation $\mu$ associated to the element  $u\in \WE$.

\vskip 4pt

Next, we focus on three special divisors on $\ocH$, see also \cite[6.8, 6.9]{alexeev15the-uniformization}:

\begin{enumerate}
\item $E_0:= E_{2:(1^{27})}$
\item The \emph{syzygetic divisor} $E\syz:= E_{2:(2^{10}, 1^7)}$.
\item The \emph{azygetic divisor} $E\azy:= E_{2:(3^6,1^9)}$.
\end{enumerate}

These three divisors correspond to the boundary divisors where
there are exactly two branch points  lying on the first irreducible
component $R_1$ and having local monodromy
$w_1, w_2\in \WE$. For $E_0$ the reflections $w_1$ and $w_2$ are equal, thus the partition associated to $w_1\cdot w_2$ equals
$\mu=(1^{27})$. For
$E\syz$ the local monodromies $w_1$ and $w_2$ are different and commuting and the associated partition is $\mu=(2^{10}, 1^7)$, whereas for $E\azy$ the reflections $w_1$ and $w_2$ do not commute, in which case the partition describing the cycle type of $w_1\cdot w_2$ is $(3^6, 1^9)$. As explained in \cite[6.6]{alexeev15the-uniformization}, we have the following relation:
\begin{equation}\label{pullbackbi}
\mathfrak{b}^*(B_i)=\sum_{\mu\in \mathcal{P}_i} \mathrm{lcm}(\mu) E_{i:\mu}.
\end{equation}

\vskip 3pt

On the space $\ohur$ we define the \emph{reduced} divisors $D_{i:\mu}$ which are the set-theoretic
images of $E_{i:\mu}$. In particular, we have the three key divisors
$D_0$, $D\syz$, $D\azy$.  By \cite[6.13]{alexeev15the-uniformization}
the pullbacks of the key divisors under the quotient map
$q\colon \och\to\ohur$ are
\begin{equation}\label{eq:divs-H-Hur}
E_0 = q^*\bigl(\frac12 D_0\bigr), \quad  E\syz = q^*(D\syz),
\quad  E\azy = q^*\bigl(\frac12 D\azy\bigr).
\end{equation}
Furthermore, $q^*(D_{i:\mu})=E_{i:\mu}$, for $i=3, \ldots, 12$ and $\mu \in \mathcal{P}_i$.

\vskip 4pt

At the level of the partial compactification $\widetilde{\mathcal{G}}_{E_6}$ \cite[9.5]{alexeev15the-uniformization} the pullbacks under
$\tilde{\beta}\colon\ohur\ratmap \widetilde{\cG}_{\EE}$ are
\begin{equation}\label{eq:divs-G-Hur}
\tilde{\beta}^*(D_{\EE}) = D_0, \quad \tilde{\beta}^*(D\syz) = D\syz, \quad
\tilde{\beta}^*(D\azy) = D\azy.
\end{equation}

For further details regarding the local description of the morphism $\tilde{\beta}$ we refer to Section \ref{subsekt}.
When carrying out divisor class calculations we will not distinguish between the spaces
$$\widetilde{\Hur}:=\Hur \cup D_0\cup D_{\mathrm{syz}}\cup D_{\mathrm{azy}}\subseteq \ohur$$ and $\widetilde{\mathcal{G}}_{E_6}$ and we will accordingly identify the divisors $D_0, D_{\mathrm{syz}}$ and $D_{\mathrm{azy}}$ on the two spaces.

\subsection{Properties of the rational map $\bm{PT}$}
\label{sec:contracted-divisors}

The Prym-Tyurin map $PT\colon \hur\rightarrow \mathcal{A}_6$ extends to a rational map
$PT\colon~\ohur~\ratmap~\overline{\mathcal{A}}_6$ for which we use the same
symbol. We denote by $U_\ohur$ the domain of definition of this
  rational map. Since $\ohur$ is normal, the complement
$\ohur\setminus U_\ohur$ has codimension at least $2$.

In Section 5.2. of \cite{alexeev15the-uniformization}, we assigned to  any point
$[\pi\colon C\to R, p_1,\dotsc, p_{24}] \in \hh$ a group
Prym-Tyurin variety $PT(C,D)=\mbox{Im}(D-1)$ for the induced endomorphism
$D$ of $JC=\mbox{Pic}^0(C)$. It is a semiabelian variety of dimension 6, that is, an
extension
\begin{displaymath}
  0 \longrightarrow T \longrightarrow PT(C, D) \longrightarrow A \longrightarrow 0
\end{displaymath}
of an abelian variety $A$ by a torus $T$.

The toric rank $\torrk := \dim T$ of the semiabelian variety $PT(C,D)$ is an upper semicontinuous function on
$\ohur$.  By \cite[Thm. 5.9]{alexeev15the-uniformization}, the domain of
definition $U_\ohur$ contains the open set $\{\torrk \le~1\}$.

\begin{lemma}\label{lem:PT-no-blowup}
The rational map $PT\colon \ohur \ratmap \overline{\mathcal{A}}_6$ does not create new
divisors. In other words, for any resolution of singularities
  \begin{center}
    \begin{tikzcd}
      X \arrow[d, "f"'] \arrow[rd, "g"] \\
      \ohur \arrow[r, dashrightarrow, "PT"] & \overline{\mathcal{A}}_6
    \end{tikzcd}
  \end{center}
  and for any closed subset $Z\subseteq \ohur$ such that
  $\codim Z\ge 2$, one has $\codim g( f\inv(Z)) \ge 2$.
\end{lemma}
\begin{proof}
We have to show that, for every irreducible subset $Z\subseteq
\ohur\setminus U_\ohur$, one has $\codim g( f\inv(Z)) \ge 2$. By the
previous paragraph, we know that $Z\subseteq \{\torrk \ge 2\}$.

By the Borel  theorem \cite[Thm. A]{borel1972some-metric}
  applied to a smooth cover of $\ohur$, the map $PT\colon \Hur\to \mathcal{A}_6$ extends
  to a regular map to the Satake-Baily-Borel compactification
  $\ohur\to \overline{\mathcal{A}}_6\sat = \mathcal{A}_6\sqcup \mathcal{A}_5\sqcup \dotsc \sqcup \mathcal{A}_0$.  Thus,
  $g(f\inv(Z))$ is contained in the preimage of
  $\mathcal{A}_4\sqcup\dotsc\sqcup \mathcal{A}_0$ under the map $\overline{\mathcal{A}}_6 \to \overline{\mathcal{A}}_6\sat$. It
  has codimension at least $2$ in~$\overline{\mathcal{A}}_6$.
\end{proof}

\begin{corollary}\label{cor:PT*-well-defined}
The divisorial pushforward map $PT_*\colon \Div(\ohur) \to \Div(\overline{\mathcal{A}}_6)$ is well
defined.
\end{corollary}

By \cite[Thm. 7.17]{alexeev15the-uniformization}, the divisors $D_0$,
$D\syz$, $D\azy$ are the only boundary divisors \emph{not contracted by} the
morphism $PT\colon U_\ohur \to \overline{\mathcal{A}}_6$. The divisor $D_0$ maps to the boundary
$D_6$ of $\overline{\mathcal{A}}_6 \setminus \mathcal{A}_6$, while $D\syz$ and $D\azy$ map onto divisors
not supported on the boundary.

We have a bijection between divisors on $\ohur$ and the divisors on the
domain of definition $U_\ohur$ of $PT$. Thus, for a divisor $D$ on
$\oasp$ we have the rational pullback divisor $PT^*(D)$ on $\ohur$
which is the closure of the corresponding regular pullback divisor on
$U_\ohur$.

\begin{definition}\label{def:junk-divs}
Denote by $\junk$ the subgroup of $\Pic(\ohur)\otimes\bQ$ generated by
the boundary divisors on $\ohur$ different from $D_0$, $D\syz$, $D\azy$.
\end{definition}

\subsection{ The Hodge classes $\bm{\lambda, \lambdam, \lambdap}$}
\label{sec:lam-lam5}

A point of $\ohur$ represents a  cover
$t:=[\pi\colon C\to R, p_1+\cdots+p_{24}]$ with $\WE$-monodromy. The Kanev correspondence $D$ on $C$ induces
an eigenspace decomposition
\begin{displaymath}\label{decomp1}
  H^0(C, \omega_C) = H^0(C, \omega_C)^{(-5)} \oplus H^0(C, \omega_C)^{(+1)}
\end{displaymath}
into subspaces of dimension $6$ and $40$ respectively. We denote by $\bE$ the Hodge bundle over $\ohur$ with fiber $H^0(C,\omega_C)$
over a point $t\in \ohur$ and by  $\bE^{(-5)}$ and
$\bE^{(+1)}$ the Hodge eigenbundles globalizing the decomposition (\ref{decomp1}), that is, having fibres $H^0(C, \omega_C)^{(-5)}$ and $H^0(C, \omega_C)^{(+1)}$ over $t$. We denote by
\begin{equation}\label{eigen1}
\lambda^{(-5)}=c_1(\mathbb E^{(-5)}) \ \ \mbox{ and }\ \ \lambda^{(+1)}:=c_1(\mathbb E^{(+1)})
\end{equation} the corresponding Hodge eigenclasses. Since $\lambda^{(-5)}=PT^*(\lambda_{1})$, determining
$\lambda^{(-5)}$ explicitly is essential for any application concerning the birational geometry of $\overline{\mathcal{A}}_6$.

\vskip 3pt

Theorem 6.17 and Remark 6.18 of \cite{alexeev15the-uniformization}
establish the following  important formula for the Hodge class on $\ohur$:
\begin{equation}\label{eq:lambda}
  \lambda = \frac{33}{46} D_0 + \frac{17}{46} D\syz
  + \frac{7}{46} D\azy  \modjunk
\end{equation}

\section{Twenty five fundamental Hodge bundles on $\ohur$.}
\label{sec:25pryms}

The main purpose of this section is to determine the Hodge classes $\lambda_1, \ldots, \lambda_{25}\in CH^1(\widetilde{\hur})$ associated to the irreducible representations of $\WE$.  In particular, we shall compute the class of the $(-5)$-Hodge eigenbundle $\lambda^{(-5)}$  and thus prove Theorem \ref{mainthm_1}.  We first describe our strategy. Theorem 6.17 of
\cite{alexeev15the-uniformization} has been used  to compute the Hodge class $\lambda\in CH^1(\ohur)$ for the universal family of degree $27$
covers, corresponding to the lines on a fixed cubic surface. In that case, $\lambda = \lambdam + \lambdap$, and the
summands of
\begin{math}
  H^0(C, \omega_C) = H^0(C, \omega_C)^{(-5)} \oplus H^0(C, \omega_C)^{(+1)}
\end{math}
are associated with irreducible representations of the Weyl group $\WE$. Namely, the $27$-dimensional representation of $\WE\hookrightarrow S_{27}$ has
character 1~+~6~+~20b, whose dimensions add up to
27. The Hodge eigenbundles $\bE_{1}$, $\bE_{6} = \bE^{(-5)}$, and
$\bE_{20b} = \bE^{(+1)}$ associated with these characters have ranks
$0+6 + 40 = 46 = g(C)$.

\subsection{}\label{num:many-curves}
The 27:1 cover $\pi\colon C\rightarrow \mathbb P^1$ whose fibres correspond to lines on a cubic surface is merely one of many. Let $\widetilde{\pi}\colon \wC\to \bP^1$ be the
Galois closure of $\pi$. Then $C = \wC / G_{27}$, where the
maximal index 27 subgroup $G_{27}$ has been introduced in \ref{sec:maxsugroups}. We have further covers associated to subgroups of $\WE$:

\begin{enumerate}
\item A maximal subgroup of index 36. The cover
  $C_{36} := \wC / G_{36} \to \bP^1$ is associated with the permutation
  representation $\WE \hookrightarrow S_{36}$ with character
  1~+~15b~+~20b. The points of the fibers of $C_{36}\to\bP^1$ correspond
  to the pairs of roots $\pm r$ of the $\WE$ root lattice;
  equivalently, to the double sixers of lines on a cubic surface.  The
  ranks of the respective vector bundles $\bE_i$ are
  $0+45+40 = 85 = g(C_{36})$.

\item A maximal subgroup of index 45. The cover
  $C_{45} := \wC / G_{45} \to \bP^1$ is associated with the permutation
  representation $\WE \hookrightarrow S_{45}$ with character
  1~+~24 + 20b. The points of the fibers of $C_{45}\to\bP^1$ correspond to
  the triangles $\{\ell_1,\ell_2,\ell_3\}$ of lines on a cubic
  surface.  The ranks of the respective vector bundles $\bE_i$ are
  $0+96+40 = 136 = g(C_{45})$.

\item More generally, for each fixed representative $w_{\alpha}$ of one of the 25 conjugacy classes in
  $\WE$, labeled as described in \ref{notation:conjclass}, recalling that $Z_{\alpha}:=Z_{w_{\alpha}}$ is the centralizer of $w_{\alpha}$,  we have the
  curve $A_{\alpha} := \wC/Z_{\alpha}$.

\item Similarly, let $W_{\alpha}=\la w_{\alpha}\ra$ be the cyclic
  subgroup generated by $w_{\alpha}$. This gives rise to 25 curves
  $B_\alpha := \wC /W_\alpha$.
\end{enumerate}

Each of these families gives a map to a certain moduli space of curves and has a Hodge bundle whose first Chern class we can
compute as a linear combination of $D_0$, $D\syz$, $D\azy$ modulo the
other boundary divisors $\junk$. Each Hodge bundle is a direct sum of isotypical components for
the 25 irreducible representations of $\WE$, that is, a direct sum of the same basic 25
Hodge bundles (with appropriate multiplicities).  The multiplicities of these isotypical components are
easily computable.  Thus, given 25 ``linearly independent'' families,
we can compute the semi-ample Chern classes $\lambda_i = c_1(\bE_i)$ of the 25 bundles $\bE_i$
labeled by the characters of~$\WE$. It turns out that the relations obtained by considering universal versions of the curves
$B_\alpha$ are linearly independent, so they work for this purpose.

In particular, this gives us a formula for $\lambda_6 = \lambda^{(-5)}$, that is,  the
first Chern class of the vector bundle we denoted $\bE^{(-5)}$
in Section~\ref{sec:lam-lam5}. We now put this program to practice.

\subsection{$\ohur$ as a moduli space of Galois admissible covers}

In what follows we choose to view $\hh$ as the moduli space of $\WE$-\emph{Galois admissible covers}
$$[\widetilde{\pi}\colon \widetilde{C}\rightarrow  R,  p_1, \ldots, p_{24}].$$
This means that $[R, p_1, \ldots, p_{24}]\in \overline {\mathcal{M}}_{0, 24}$, as usual, $\widetilde{\pi}^{-1}(R_{\mathrm{sing}})=\widetilde{C}_{\mathrm{sing}}$ and that there is a $\WE$-action on $\widetilde{C}$ compatible with $\widetilde{\pi}$ such that the restriction
$$\widetilde{\pi}\colon \widetilde{\pi}^{-1}\bigl(R_{\mathrm{reg}}\setminus \{p_1, \ldots, p_{24}\}\bigr)\rightarrow R_{\mathrm{reg}}\setminus \{p_1, \ldots, p_{24}\}$$ is a principal $\WE$-bundle. At each node $q\in C_{\mathrm{sing}}$, the action of the stabilizer $\mathrm{Stab}_q\bigl(\WE\bigr)\subseteq \WE$ is \emph{balanced}, that is, the eigenvalues of the actions on the tangent spaces on the two branches of the tangent spaces of $\widetilde{C}$ at $q$ are multiplicative inverses to one another.

To recover the description of $\hh$ given in (\ref{sec:H-Hur_GE6}), we fix the subgroup $G_{27}=\mbox{Stab}_{\WE}(a_6)\subseteq \WE$ and note that if $\widetilde{\pi}\colon \widetilde{C}\rightarrow  R$ is
a $\WE$-Galois cover, then $\pi:=\pi_{G_{27}}\colon \widetilde{C}/G_{27}\rightarrow R$ is a degree $27$ cover with monodromy group equal to $\WE$.
The inverse operation is obtained by taking the Galois closure of each degree $27$ cover $\pi\colon C\rightarrow R$ with $\WE$-monodromy. Both of these operations can be carried out in families.

\begin{notation}
For a Galois $\WE$-cover  $\widetilde{\pi}\colon \widetilde{C}\to R$ and for a subgroup $G\subseteq \WE$, we denote $C_G:=\widetilde{C}/G$ and $\pi_G
\colon C_G\rightarrow R$ the induced cover of degree $d = [\WE:G]$. We further set $g_G:=p_a(C_G)$.
\end{notation}

\begin{lemma}
The arithmetic genus $g_G$ of the curve $C_G$ is
  \begin{equation}
    \label{eq:genus}
    g_G = 12 a_{2c} - d +1
  \end{equation}
  where $d = [\WE : G]$ and $a_{2c}$ is given by
  Equation~\eqref{eq:cycle-type} for $u$ in the conjugacy class 2c
  containing the reflections of $\WE$.
\end{lemma}
\begin{proof}
The sheets of the cover $\pi_G\colon C_G\rightarrow \mathbb P^1$ over a general point from $\mathbb P^1$ are in bijection with the set of cosets
$\WE/G$. The monodromy action by an element $u\in \WE$ is given by multiplication $xG\mapsto uxG$ on the set of cosets. If $[\pi_G\colon C_G\rightarrow \mathbb P^1, p_1, \ldots, p_{24}]$ corresponds to a general element from $\hh$, then $\pi_G$ is ramified over each of the $24$ points $p_i$ according to the ramification profile $2^{a_{2c}} 1^{a_{2c}}$, where $a_{2c}$ and $b_{2c}$ have been defined in (\ref{vector3}).
Applying the Hurwitz formula to $\pi_G$, we thus have $2g_G-2 = d(-2) + 24a_{2c}$, which finishes the proof.
\end{proof}

\vskip 4pt

\subsection{Computation of Hodge classes on $\hh$.}
Having fixed a subgroup $G\subseteq \WE$ of index $d$,  the assignment
$[\widetilde{\pi}\colon \widetilde{C}\rightarrow \mathbb P^1, \ p_1, \ldots, p_{24}]\mapsto [C_G]$ induces a regular map  $$\och\to\overline{\mathcal{M}}_{g_G}$$ and accordingly a  Hodge bundle $\bE_G$ on $\hh$ of rank $g_G$ obtained by pulling back the Hodge bundle from $\overline{\mathcal{M}}_{g_G}$. We aim to compute its determinant $\lambda_G := c_1(\bE_G)$ on $\och$. To that end we need some preparation:

\vskip 3pt

The universal stable curve over $\overline{\mathcal{M}}_{0,24}$ is denoted by
$\pi_{25}\colon \overline{\mathcal{M}}_{0,25} \to \overline{\mathcal{M}}_{0,24}$ and forgets the marked
point labeled by $25$. We recall the following standard formulas, see for instance
\cite{farkas2003mori-cone}.

\begin{equation}\label{eq:fg1}
  c_1(\omega_{\pi_{25}})=\psi_{25}-\sum_{i=1}^{24} \delta_{0:i,
     25}\in CH^1(\overline{\mathcal{M}}_{0, 25}).
 \end{equation}
 \begin{equation}\label{eq:fg2}
   \sum_{i=1}^{24} \psi_i= \sum_{i=2}^{12}\frac{i(24-i)}{23}
     [B_i]
      \in CH^1(\overline{\mathcal{M}}_{0, 24}); \quad
     \kappa_1=
     \sum_{i=2}^{12}\frac{(i-1)(23-i)}{23}[B_i]
\end{equation}
Here $\psi_i$ are the cotangent tautological classes corresponding to the marked points, whereas $\kappa_1$ is the usual $\kappa$-class.

\begin{theorem}\label{lam}
Let $G$ be a subgroup of $\WE$ as before. Assume the ramification profile of the degree $d$ cover
$C_G\rightarrow \mathbb P^1$ corresponding to a general element $[\widetilde{C}\rightarrow \mathbb P^1, p_1, \ldots, p_{24}]\in \hh$
over each of the $24$ branch points $p_i$ is of the type $2^a1^b$, where $2a+b=d$. Then the Hodge
  class $\lambda_G$ on  $\hh$ is given by
\begin{displaymath}
  \lambda_G = \sum_{i=2}^{12}\sum_{\mu\in \mathcal{P}_{i}}\frac{1}{12}
  \lcm(\mu) \Bigl(\frac{3a}{2} \frac{i(24-i)}{23}
  - d +\frac{1}{\mu}\Bigr) \, [E_{i:\mu}]\in CH^1(\hh).
\end{displaymath}
\end{theorem}

\begin{proof}
The proof follows the lines of that of
\cite[Theorem 6.17]{alexeev15the-uniformization}, with appropriate
  changes we indicate below. Over the Hurwitz space $\hh$ we consider the universal
  $\WE$-admissible cover $f\colon \cC_G \rightarrow \cP$ of degree $d$, where
  $$\cP:=\hh\times_{\overline{\mathcal{M}}_{0, 24}} \overline{\mathcal{M}}_{0, 25}$$ is the universal
  degree $d$ \emph{orbicurve} of genus zero over $\hh$. We fix a general point
  $$t=[\pi_G\colon C_G\rightarrow R,  p_1, \ldots, p_{24}]$$ of a boundary divisor
  $E_{i:\mu}$, where
  $\mu=(\mu_1, \ldots, \mu_{\ell})\in \mathcal{P}_i$.  In particular, $R$ is the union of two smooth rational curves $R_1$ and $R_2$ meeting at a point $q$. The local ring of the space of Harris-Mumford admissible covers has the
  the following local description at $t$:
  \begin{equation}\label{localring2}
    \mathbb C
    [[t_1, \ldots, t_{21}, s_1, \ldots,
    s_{\ell}]]/s_1^{\mu_1}=\cdots=s_{\ell}^{\mu_{\ell}}=t_1,
  \end{equation}
  where $t_1$ is the local parameter on $\overline{\mathcal{M}}_{0, 24}$ corresponding to
  smoothing the node $q\in~R$.  The space $\mathcal{P}$ has a singularity of
  type $A_{\mathrm{lcm}(\mu)-1}$, and accordingly $\cC_G$ has singularities of
  type $A_{\mathrm{lcm}(\mu)/\mu_i-1}$ at the $\ell$ points
  corresponding to the inverse image of $R_{\mathrm{sing}}$. Indeed, to determine the local ring of $\hh$ at the point $t$, one normalizes
  the ring (\ref{localring2}). To that end, we introduce a further parameter $\tau$ and choose primitive $\mu_j$-th roots of unity $\zeta_j$ for $j=1, \ldots, \ell$. These choices correspond to specifying the stack structure of the cover $C_G \rightarrow R$ at the points of $C_G$ lying over the point $q\in R_{\mathrm{sing}}$. Thus
  $$\widehat{\mathcal{O}}_{[t, \zeta_1, \ldots, \zeta_{\ell}], \ \hh}=\mathbb C [[t_1, \ldots, t_{21}, \tau]]$$
  and  $s_j=\zeta_j  \tau^{\frac{\mathrm{lcm}(\mu)}{\mu_j}}$, for $j=1, \ldots, \ell$. Accordingly,  the map $\mathfrak{b}\colon \hh\rightarrow \overline{\mathcal{M}}_{0,24}$ is branched with order $\mathrm{lcm}(\mu)$ at each such point $[t, \zeta_1, \ldots, \zeta_{\ell}]$. When the stack data $(\zeta_1, \ldots, \zeta_{\ell})$ is clear from the context, we drop it and we write as before $t=[t, \zeta_1, \ldots, \zeta_{\ell}]\in \hh$
  when referring to a point of $\hh$.

\vskip 4pt

  Let $\phi\colon \cP\rightarrow \hh$ and $\oq\colon \cP\rightarrow \overline{\mathcal{M}}_{0, 25}$
  be the two projections and put $v:=\phi\circ f\colon \cC_G  \rightarrow \hh$
  respectively $\of:=\oq\circ f\colon \cC_G \rightarrow \overline{\mathcal{M}}_{0, 25}$. Note that $v$ respectively $\of$ are viewed as the universal curve of genus $g_G$ over $\hh$ and $\overline{\cM}_{0,25}$ respectively. The ramification
  divisor of $f$ decomposes as
  $$\mathrm{Ram}(f)=R_1+\cdots +R_{24} \subseteq  \cC_G,$$
  where a general point of $R_i$ is of the form
  $[\pi\colon C_G \rightarrow R, \ p_1, \ldots, p_{24}, x]$, with $R$ being a nodal rational curve and $x\in C$ being one of
  the $a$ ramification points lying over the branch point $p_i$. Since over each branch point lie $a$ ramification points, we have
 $f_*([R_i])=a[\mathfrak B_i]$, where
  $\mathfrak B_i\subseteq \cP$ is the corresponding branch divisor.

\vskip 4pt

We apply  the Riemann-Hurwitz formula to the finite map $f\colon \mathcal{C}_G\rightarrow \mathcal{P}$. Accordingly, we can write
  $c_1(\omega_v)=f^*\oq^* c_1(\omega_{\pi_{25}})+[\mathrm{Ram}(f)]$, where we recall that $\pi_{25}\colon \overline{\cM}_{0,25}\rightarrow \overline{\cM}_{0,24}$ is the morphism forgetting the last marked point.  We square this identity and then push it forward via $v$ to obtain a relation in $CH^1(\hh)$. We have that
  \begin{displaymath}
  v_* c_1^2(\omega_v)=v_*\Bigl(\of^*c_1^2(\omega_{\pi_{25}})+2\of^*
  c_1(\omega_{\pi_{25}})\cdot [\mathrm{Ram}(f)]+[\mathrm{Ram}(f)]^2\Bigr).
  \end{displaymath}

  We evaluate each term, starting with the second one. We write
  $v_*\Bigl(\of^*c_1(\omega_{\pi_{25}})\cdot [\mathrm{Ram}(f)]\Bigr)=$
  \begin{displaymath}
    \sum_{i=1}^{24}
    \phi_* \Bigl( \oq^* c_1\left(\omega_{\pi_{25}}\right)\cdot a[\mathfrak
    B_i] \Bigr) =
    a \sum_{i=1}^{24}\phi_*\oq^* \Bigl(c_1(\omega_{\pi_{25}})\cdot
    [\Delta_{0:i, 25}]\Bigr) =
    a\,\mathfrak{b}^*\Bigl(\sum_{i=1}^{24} \psi_i\Bigr).
  \end{displaymath}

  Furthermore, we write
  $f^*(\mathfrak{B}_i)=2R_i+A_i$, where the residual divisor
  $A_i$ defined by the previous equality maps
  $b:1$ onto $\mathfrak{B}_i$. Note that $A_i$ and
  $R_i$ are disjoint, hence $f^*([\mathfrak{B_i}])\cdot
  R_i=2R_i^2$. Therefore
  \begin{displaymath}
  v_*([R_i]^2)=
  \frac{a}{2} \phi_*([\mathfrak{B}_i^2])=
  \frac{a}{2} \phi_*(\oq^*\bigl(\delta_{0:i, 25}^2)\bigr)=
  - \frac{a}{2} \mathfrak{b}^*(\psi_i).
  \end{displaymath}
  Using Equation \eqref{eq:fg2}, we compute that
  \begin{displaymath}
    v_*\bigl([\mathrm{Ram}(f)]^2\bigr)=
    v_*\Bigl(\sum_{i=1}^{24} [R_i]^2\Bigr)=-\frac{a}{2} \mathfrak{b}^*\Bigl(\sum_{i=1}^{24}\psi_i\Bigr)=
    -\frac{a}{2} \sum_{i=2}^{12}\frac{i(24-i)}{23} \mathfrak{b}^*([B_i]).
  \end{displaymath}
  We use Equation \eqref{eq:fg1}, and the relation $\pi_*(\delta_{0:i,
    25}^2)=-\psi_i$ for $i=1, \ldots, 24$, to write:
  \begin{eqnarray*}
    &&v_*\of^*c_1^2(\omega_{\pi_{25}})=\phi_*\Bigl(d\,
       \oq^*c_1^2(\omega_{\pi_{25}})\Bigr)=
       d\, \mathfrak{b}^*\pi_*\Bigl(\psi_{25}-\sum_{i=1}^{24}
    \delta_{0:i, 25}\Bigr)^2= \\
    &&= d\, \mathfrak{b}^*\Bigl(\kappa_1-\sum_{i=1}^{24} \psi_i\Bigr)=
    -d\, \mathfrak{b}^*\Bigl(\sum_{i=2}^{12} [B_i]\Bigr),
  \end{eqnarray*}
where the last equation is again a consequence of (\ref{eq:fg2}).

We find the following expression for the pull-back of the Mumford
$\kappa$ class to $\hh$:
\begin{equation}\label{mumf1}
  v_*c_1^2(\omega_v)\equiv \sum_{i=2}^{12}
  \Bigl(\frac{3a}{2} \frac{i(24-i)}{23}- d\Bigr)
  \mathfrak{b}^*(B_i)\equiv
  \sum_{i=2}^{12}\sum_{\mu\in \mathcal{P}_i}
  \mathrm{lcm}(\mu)\Bigl( \frac{3a}{2} \frac{i(24-i)}{23}- d \Bigr)E_{i:\mu}.
\end{equation}
Via a Grothendieck-Riemann-Roch calculation in the case of the universal genus
$g_G$ curve $v\colon \cC_G\rightarrow
\hh$, coupled with the local analysis of the fibers of the branch map
$\mathfrak{b}$, we find
\begin{displaymath}
  12\lambda_G=
  v_*c_1^2(\omega_v)+\sum_{i=2}^{12}\sum_{\mu\in \mathcal{P}_i}
  \mathrm{lcm}(\mu)\cdot \frac{1}{\mu}\
  [E_{i:\mu}].
\end{displaymath}

Substituting in (\ref{mumf1}), we finish the proof.
\end{proof}

We now make Theorem \ref{lam} more precise involving the monodromy vectors defined in (\ref{vector3}).

\begin{corollary}\label{cor:lambda-hur}
Let $G$ be a subgroup of $\WE$ of index $d$ and let $\WE\into S_d$ be the monodromy
action for a generic cover $[\pi\colon C_G \to \bP^1, p_1, \dotsc, p_{24}]$ in this family.
 Suppose that the cycle types of the elements $\alpha\in \WE$ in the
  conjugacy classes 2c, 2b, 3b are
  $2^{a_{2c}}1^{b_{2c}}$,\  $2^{a_{2b}}1^{b_{2b}}$ and
  $3^{a_{3b}}1^{b_{3b}}$ respectively.  Then the Hodge class $\lambda_G$ on $\ohur$ is:
  \begin{equation}\label{eq:lambda-hur}
    \lambda_G =
    \frac{11a_{2c}}{92} [D_0] +
    \frac16 \Bigl(\frac{66a_{2c}}{23}
    - \frac{3a_{2b}}2 \Bigr) [D\syz] +
    \frac18 \Bigl(\frac{66a_{2c}}{23}
    - \frac{8a_{3b}}3 \Bigr) [D\azy] \modjunk.
  \end{equation}
\end{corollary}
\begin{proof}
  For the divisors $E_0, E\syz, E\azy$ one has $i=2$.  The classes 2c,
  2b, 3b are the conjugacy classes respectively of a reflection $w$, a
  product of two commuting reflections $w_1\cdot w_2$ and two non commuting
  reflections $w_1\cdot w_2$. Over $D_0$, respectively $D\syz$, $D\azy$, we compute
  $d - \frac1{\mu}$ to be respectively $0$, $\frac{3a_{2b}}2$,
  $\frac{8a_{3b}}3$, and $\lcm(\mu)$ to be 1, 2, 3. Finally,
  we use the relation between $E$'s and $D$'s from
  Equation~\ref{eq:divs-H-Hur}.
\end{proof}

\begin{example}
  For the maximal subgroup $G_{27}\subseteq \WE$, using $(a_{2c}, a_{2b}, a_{3c}) = (6,10,6)$ we
  recover the formula for $\lambda_{G_{27}}=\lambda$ given in Theorem \cite[Theorem 6.17]{alexeev15the-uniformization}.
\end{example}

\subsection{Prym-Tyurin varieties via Galois covers}
We now discuss a different representation-theoretic interpretation of the Prym-Tyurin variety $PT(C,D)$ associated to a $\WE$-cover $\pi\colon C\rightarrow \mathbb P^1$. Recall that in \ref{sec:maxsugroups} we fixed the maximal index $27$  subgroup $G_{27}=\mbox{Stab}_{\WE}(a_6)$ of $\WE$.
For a $\WE$-Galois cover $[\widetilde{\pi}\colon \widetilde{C}\rightarrow R, p_1+\cdots+p_{24}]$, we denote by $\pi\colon C=\widetilde{C}/G_{27}\rightarrow R$ the associated degree $27$ cover with monodromy group $\WE$.  Let $(\EE)_{\mathbb{C}}:=\EE \otimes \mathbb{C}$. Notice that $(\EE)_{\mathbb{C}}$  is also generated by the elements of the orbit  of $a_6$ (all weights of $\EE$).
Following \cite[5.1]{donagi1993decomposition-of-spectral},  we define the \emph{Prym variety} associated to the lattice $\EE$ as the abelian variety parametrizing equivariant maps to $J\widetilde{C}$, that is,
$$\mbox{Prym}_{\EE}(J\widetilde{C}):=\mbox{Hom}_{\WE}\bigl((\EE)_{\mathbb{C}}, J\widetilde{C}\bigr).$$

The evaluation at the element $a_6$ induces an injective morphism of abelian varieties
(\cite[Lemma 5.4.]{LP} and  \cite[Proposition 5.2.]{LP})
$$\mbox{eval}_{a_6}\colon \mbox{Hom}_{\WE}\bigl(\EE, J\widetilde{C}\bigr)\hookrightarrow J\widetilde{C}, \ \ [\upsilon \colon \EE\rightarrow J\widetilde{C}]\mapsto \upsilon(a_6).$$
In this way $\mbox{Prym}_{\EE}(J\widetilde{C})$ is endowed with a polarization.
The image of the map $\mbox{eval}_{a_6}$ above lands inside $JC=\bigl(J\widetilde{C}\bigr)^{G_{27}}$. We now summarize  results from \cite[Section 12]{donagi1993decomposition-of-spectral}, see also \cite[Section 5]{LP}:

\begin{theorem}\label{thm:prym-galois}
The evaluation induces an isomorphism of $6$-dimensional ppav $\mathrm{Prym}_{\EE}(J\wC)\cong PT(C,D)$.
\end{theorem}

Since the proof given in \cite[Section 12]{donagi1993decomposition-of-spectral} is representation-theoretical it works without modification in families. Passing to tangent spaces at the origin, Theorem \ref{thm:prym-galois} implies that one has a natural isomorphism of vector spaces
\begin{equation}\label{eq:tgspaces}
\mbox{Hom}_{\WE}\bigl(E_6, H^0(\widetilde{C}, \omega_{\widetilde{C}})\bigr)\cong H^0(C, \omega_C)^{(-5)}.
\end{equation}

\subsection{Computing the 25 fundamental Hodge classes}
\label{sec:compute-25}

We denote by $\rho_1, \ldots, \rho_{25}$ the irreducible representations of $\WE$. We also fix a subgroup $G\subseteq \WE$ of index $d$. For each $\WE$-Galois cover $[\tilde{\pi}\colon \widetilde{C}\rightarrow \mathbb P^1, p_1, \ldots, p_{24}]$, the space of differentials $H^0(\widetilde{C}, \omega_{\widetilde{C}})$ is a $\WE$-module and accordingly we have the following decompositions into sums of irreducible representations:
\begin{equation}
  \label{eq:H0-decomps}
  H^0(\widetilde{C}, \omega_{\widetilde{C}}) =
  \bigoplus_{i=1}^{25} \rho_i \otimes \mathrm{Hom}_{\WE}\Bigl(\rho_i, H^0(\widetilde{C}, \omega_{\widetilde{C}})\Bigr) ,
  \ \ \
  H^0(C_G, \omega_{C_G}) =
  \bigoplus_{i=1}^{25} \rho_i^{G} \otimes \mathrm{Hom}_{\WE}\Bigl(\rho_i, H^0(\widetilde{C}, \omega_{\widetilde{C}})\Bigr).
\end{equation}

\begin{notation}
We denote by $\widetilde{\bE}$ the $\WE$-Hodge bundle on $\hh$, that is, having fibre
$H^0(\widetilde{C}, \omega_{\widetilde{C}})$ over a point $[\widetilde{\pi}\colon \widetilde{C}\rightarrow R]\in \ohur$.
\end{notation}

We now define Hodge bundles corresponding to each irreducible representation of $\WE$.
\begin{definition}
For each $i=1, \ldots, 25$, let
$\bE_i := \Hom_{\WE} \bigl(\rho_i, \widetilde{\bE}\bigr)$ regarded as a vector bundle on $\ohur$.  We
let $\lambda_i :=c_1(\bE_i)\in CH^1(\ohur)$.
\end{definition}

We have therefore the following identity in the $K$-group of $\ohur$:

\begin{equation}\label{hodge-kgroup}
\widetilde{\bE}=\bigoplus_{i=1}^{25} \rho_i \otimes \bE_i.
\end{equation}

The dimensions of the invariant subspaces $\rho_i^G$ as usual are
given by the formula
  \begin{equation}
    \label{eq:mults}
    \dim(\rho_i^{G} ) = \frac1{|G|} \
    {\sum_{g\in G} \Tr_{\rho_i}(g)}.
  \end{equation}
  Here, for $g\in \WE$ in
  the conjugacy class $\alpha$, we have $\Tr_{\rho_i}(g)=\Tr_{\chi_i}(\alpha)$ in the character table of $\WE$, see
  Table~\ref{tab:E6-chartable}.

\vskip 4pt

We now come to the first main result of this paper, the explicit computation of all the classes $\lambda_i$. This implies Theorem \ref{mainthm_1}.

\begin{theorem}\label{thm:25-curves}
  The ranks $\rk(\bE_i)$ and the 25 fundamental Hodge classes $\lambda_i = c_1(\bE_i)$ on $\ohur$ in terms of  the generators $D_0,D\syz,D\azy$
  $\modjunk$ are given as in Table~\ref{tab:rk-lambda}.
\end{theorem}

\begin{table}[htp!]\centering
\begin{tabular}{|rl|c|ccc|ccc|} \hline
$\chi$ & name & $\rk\bE_i$ & $D_0$ & $D\syz$ & $D\azy$ & $a_{2c}$ & $a_{2b}$ & $a_{3b}$  \\
\hline
1 & 1 & 0 & 0 & 0 & 0 & 0 & 0 & 0 \\
2 & \tov{1} & 11 & 11/92 & 11/23 & 33/92 & 1 & 0 & 0 \\
3 & 10 & 50 & 55/92 & 32/23 & 127/276 & 5 & 4 & 4 \\
4 & 6 & 6 & 11/92 & $-1/46$ & 7/276 & 1 & 2 & 1 \\
5 & \tov{6} & 54 & 55/92 & 87/46 & 403/276 & 5 & 2 & 1 \\
6 & 20a & 100 & 55/46 & 41/23 & 73/46 & 10 & 12 & 6 \\
7 & 15a & 45 & 55/92 & 9/23 & 127/276 & 5 & 8 & 4 \\
8 & \tov{15a} & 105 & 55/46 & 64/23 & 311/138 & 10 & 8 & 4 \\
9 & 15b & 45 & 55/92 & 41/46 & 35/276 & 5 & 6 & 5 \\
10 & \tov{15b} & 105 & 55/46 & 151/46 & 265/138 & 10 & 6 & 5 \\
11 & 20b & 40 & 55/92 & 9/23 & 35/276 & 5 & 8 & 5 \\
12 & \tov{20b} & 160 & 165/92 & 119/23 & 1025/276 & 15 & 8 & 5 \\
13 & 24 & 96 & 55/46 & 41/23 & 127/138 & 10 & 12 & 8 \\
14 & \tov{24} & 144 & 77/46 & 85/23 & 325/138 & 14 & 12 & 8 \\
15 & 30 & 90 & 55/46 & 59/46 & 27/46 & 10 & 14 & 9 \\
16 & \tov{30} & 210 & 55/23 & 279/46 & 96/23 & 20 & 14 & 9 \\
17 & 60a & 300 & 165/46 & 169/23 & 473/138 & 30 & 28 & 22 \\
18 & 80 & 400 & 110/23 & 210/23 & 346/69 & 40 & 40 & 28 \\
19 & 90 & 450 & 495/92 & 219/23 & 565/92 & 45 & 48 & 30 \\
20 & 60b & 240 & 275/92 & 114/23 & 181/92 & 25 & 28 & 21 \\
21 & \tov{60b} & 360 & 385/92 & 224/23 & 511/92 & 35 & 28 & 21 \\
22 & 64 & 224 & 66/23 & 80/23 & 134/69 & 24 & 32 & 20 \\
23 & \tov{64} & 416 & 110/23 & 256/23 & 530/69 & 40 & 32 & 20 \\
24 & 81 & 351 & 99/23 & 309/46 & 90/23 & 36 & 42 & 27 \\
25 & \tov{81} & 459 & 495/92 & 507/46 & 657/92 & 45 & 42 & 27 \\
\hline \end{tabular}
\medskip\caption{$\chi_i$, $\rk\bE_i$, $\lambda_i$, and $(a_{2c},a_{2b},a_{3b})(\chi_i)$}
\label{tab:rk-lambda}
\end{table}

\begin{proof}
  We apply the above formulas to the 25 cyclic groups
  $G = W_\alpha = \la w_\alpha\ra$ generated by
  25 fixed representatives $w_\alpha$ of the conjugacy classes of
  $\WE$.  Precisely, we have
  $$\lambda_G=\sum_{i=1}^{25} \mbox{dim}(\rho_i^G) \lambda_i.$$
  From \eqref{eq:mults} we compute the
  $25\times 25$ matrix of multiplicities $M=\dim(\rho_i^{W_\alpha})_{1\leq i, \alpha\leq 25}$
  and find its determinant to be $400771988324352 \ne0$, so it is
  invertible.

  We compute the vector of genera of the curves $B_\alpha = \wC/W_\alpha$ by
  \eqref{eq:genus}. Multiplying this vector by $M\inv$ we find the
  ranks of $\bE_i$.
  Next, for each of the curves $B_\alpha$, we find the 6-tuple
  $(a_{2c},b_{2c};\, a_{2b},b_{2b};\, a_{3c},b_{3c})$ by applying
  \eqref{eq:cycle-type} to the elements $u$ lying in the conjugacy
  classes 2c, 2b, 3b. Then, using Corollary~\ref{cor:lambda-hur}, we find the corresponding lambda class
  $\lambda_{W_\alpha}$ on $\ohur$. Finally, we multiply the $3\times 25$
  matrix of these lambda classes by $M\inv$ to get the expressions for
  $\lambda_i$ in terms of $D_0,D\syz,D\azy \modjunk$.

\end{proof}



\begin{remark}
  Since $\lambda = \lambdam + \lambdap$, Equation \ref{eq:lam5} and Theorem \ref{mainthm_1} are
  equivalent. There are similar identities to \ref{eq:lam5} for
  the universal covers of degree $36$ and $45$ from \ref{num:many-curves}(1,2).
\end{remark}

\begin{remark}
  From Corollary~\ref{cor:lambda-hur} we see that the Hodge class
  $\lambda_G$ is a linear function of the vector
  $\vec a=(a_{2c}, a_{2b}, a_{3b})$ given
  by an invertible matrix.  It follows that
  $\vec a$ is a linear function of the vector
  $\lambda_G$.  Associating to a cover $C_G = \wC/G$ the element
  $\sum_i (\dim\rho_i^{G}) \chi_i$ in the character space of $\WE$, we
  see that
  \begin{displaymath}
    a_\alpha(C_G) = \sum_{i=1}^{25} (\dim\rho_i^{G}) a_\alpha(\chi_i)
    \quad \text{for } \alpha = 2c,2b,3b.
  \end{displaymath}
  Then $a_\alpha(\chi)$ can be computed using the same linear algebra,
  from Equations \ref{eq:mults} and \ref{eq:cycle-type}. We list them
  in the last three columns of Table~\ref{tab:rk-lambda}.

  The following is also easy to see, cf. \eqref{eq:genus}. For any
  character $\chi$ one has
  \begin{equation}
    \label{eq:rank-to-2c}
    g(\chi):=\rank\bE(\chi) = 12 a_{2c}(\chi) - \chi(1a) + \mult_1(\chi),
  \end{equation}
  where $\chi(1a) = \dim V_\chi$ is the dimension of the
  representation, and $\mult_1(\chi)$ is the multiplicity of the
  trivial representation 1 in $\chi$. For example $g(C_{27}) = 12\cdot
  6-27+1 = 46$, and $\rank(\bE_6) = 12 \cdot 1 - 6 = 6$.
\end{remark}

\begin{remark}
  From Table~\ref{tab:rk-lambda} one can observe that for any
  character $\chi$ one has
  \begin{displaymath}
    \lambda(\chi\otimes{\ov{1}}) = \lambda(\chi) + \chi(2c)\lambda(\ov{1}),
    \qquad
    \vec a(\chi\otimes{\ov{1}}) =
    \vec a(\chi) + \chi(2c)(1,0,0).
  \end{displaymath}
\end{remark}

\section{The Weyl-Petri divisor and the ramification of the Prym-Tyurin map}\label{sec:ram}

In \cite[Section 10]{alexeev15the-uniformization}, we showed that, if a smooth $W(E_6)$-cover $[\pi \colon C \rra R, p_1+\cdots+p_{24}] \in \hur$ lies in the ramification locus of $PT$, the line bundle $L$ associated to $\pi$ satisfies $h^0(C, L) =2$ and the Petri map
\[
H^0 (C, L) \otimes H^0 (C, \omega_C \otimes L^{-1}) \lra H^0 (C, \omega_C)^{(+1)}
\]
is an isomorphism, then the Prym-Tyurin canonical image of $C$ is contained in a quadric. In this section we refine the above result by showing that the ramification divisor of $PT$ is contained in the union of two divisors $\fM$ and $\fN$ which we shall describe. In this section, we work on an alternative compactification $\widetilde{\cG}_{\EE}$ of $\Hur$ which we first discuss in some detail.

\subsection{The parameter space $\cG_{\EE}$}\label{subsekt} In \cite[9.4]{alexeev15the-uniformization} we introduced  the stack $\cG_{\EE}$ classifying $SL(2)$-equivalence classes of finite maps $[\pi \colon C \rightarrow \mathbb P^1]$ with $\WE$ monodromy, where $C$ is an irreducible curve of genus
$46$. To construct $\cG_{\EE}$, we let $\mathcal{X}_{\EE}$ denote the substack of the moduli stack $\overline{\cM}_{46}(\mathbb P^1, 27)$ parametrizing finite stable maps $\pi\colon C\rightarrow \mathbb P^1$, from an irreducible nodal curve $C$ of genus $46$ and having monodromy group $M_{\pi}$ contained in $\WE$. Then we set
$$\cG_{\EE}:=\bigl[\cX_{\EE}/SL(2)\bigr],$$
where $SL(2)$ acts on the base by linear transformations.

Let $f_{E_6}\colon \mathcal C_{E_6}\rightarrow \cG_{\EE}$ be the universal curve of genus $46$. One has a birational map
$\beta\colon
\widetilde{\hur} \dashrightarrow \cG_{\EE}$. We recall the effect of this map on the boundary divisors $D_0, D_{\mathrm{syz}}$ and $D_{\mathrm{azy}}$ of
$\widetilde{\hur}$. We fix  a point
$$t=[\pi\colon C=C_1\cup C_2\rightarrow R=R_1\cup_q R_2,\  p_1+\cdots+p_{24}]\in \widetilde{\hur},$$ where we assume that
$R_1$ and $R_2$ are smooth rational curves meeting at $q$ and that $p_1, \ldots, p_{22}\in R_1\setminus \{q\}$ whereas $p_{23}, p_{24}\in R_2\setminus \{q\}$.

If $t$ represents a general point of $D_0$, then $C_1$ is a smooth curve of genus $40$. The curve $C_2$ consists of $21$ components, of which $6$ map with degree $2$ onto $R_2$ and meet $C_1$ in two points, whereas the remaining $15$ map isomorphically onto $R_2$ and meet $C_1$ in one point. Then $\beta(t)=[\opi : \oC \rightarrow R_1]\in \cG_{\EE}$, where $\oC$ is the $6$-nodal curve obtained from $C_1$ by pairwise identifying the six pairs of points lying on the components of $C_2$ mapping 2-to-1 onto $R_2$, and $\opi$ is induced by $\pi$. If $\nu\colon C_1\rightarrow \oC$ is the normalization map, then $\overline{L} := \opi^* \cO_{R_1} (1) \in W^1_{27}(\oC)$ is uniquely characterized by the property $\nu^*(\oL)=\pi_{|R_1}^*(\cO_{C_1}(1))\in W^1_{27}(C_1)$.

\vskip 3pt

If $t$ represents a general point of $D_{\mathrm{azy}}$, then $C_1$ is smooth of genus $46$ and $\pi_{|C_1}\colon C_1\rightarrow R_1$ is a map of degree $27$ with $6$ ramification points of index $3$ over the point $q\in R_1$. Then $$\beta(t)=[\pi_{|C_1}\colon C_1 \rightarrow R_1]\in \cG_{\EE}$$ and $L_1:=\pi_{|C_1}^*(\cO_{R_1}(1))\in W^1_{27}(C_1)$.

\vskip 3pt

The case when $t$ corresponds to a general point of $D_{\mathrm{syz}}$ requires care. Then $C_1$ is a smooth curve of genus $45$. The permutations in $S_{27}$ corresponding to the roots $w_{23}$ and $w_{24}$ describing the local monodromy around $p_{23}$ and $p_{24}$ share four elements. For instance, using the standard notation for the lines on a cubic surface, we may assume $w_{23}=\alpha_{\mathrm{max}}=2h-a_1-\cdots -a_6$ and $w_{24}=\alpha_{12}=a_1-a_2$:

$$\alpha_{\mathrm{max}}= \bigl(\begin{smallmatrix}
    a_1 & a_2 & a_{3} & a_{4} & a_{5} & a_{6} \\
      b_1 & b_2 & b_3 & b_4 &  b_5  & b_6
  \end{smallmatrix}\bigr)
\  \ \mbox{ and }  \ \ \alpha_{12}= \bigl(\begin{smallmatrix}
    a_1 & b_1 & c_{13} & c_{14} & c_{15} & c_{16} \\
    a_2 & b_2 & c_{23} & c_{24} & c_{25}  & c_{26}
  \end{smallmatrix}\bigr).
$$

The curve $C_1$ meets a smooth rational component of $E$ of $C_2$ at two points $p_1$ and $p_2$  corresponding to the sheets labelled by the transpositions $(a_1, b_2)$ and $(b_1, a_2)$  corresponding to multiplying $\alpha_{\mathrm{max}}$ and $\alpha_{12}$. The map $\pi_{|E}\colon E \rightarrow R_2$ is of degree $4$ and $\pi_{|E}^*(q)=2p_1+2p_2$. We have $\beta(t)=[\opi : \oC \rightarrow R_1]$, where $\oC$ is obtained from $C_1$ by identifying the points $p_1$ and $p_2$ and $\opi$ is induced by $\pi$. Therefore
$\oC$ is an irreducible $1$-nodal curve of genus $46$. The line bundle $\oL := \opi^* \cO_{R_1}(1)\in W^1_{27}(\oC)$ is characterized by the fact that if $\nu\colon C_1\rightarrow \oC$ is the normalization map, then $\nu^*(\oL)=L_1:=\pi_{| C_1}^*(\cO_{R_1}(1))$. Moreover, if $\oC_{\mathrm{sing}}=\{z\}$, that is, $\nu^{-1}(z)=\{p_1, p_2\}$, then

$$h^0\bigl(C_1, L_1(-2p_1-2p_2)\bigr)\geq 1.$$

Because the points $p_1$ and $p_2$ are \emph{ramification} points of $L_1$, it follows that the local equations of $\cG_{\EE}$ around $t\in D_{\mathrm{syz}}$ are
$$(u,v, t_1, t_2, \ldots, t_{21}), \ u^2=v^2=t_1, $$
see \cite[Corollary 4.16]{Va}  for a similar discussion.
The parameters $t_1, \ldots, t_{21}$ correspond to deforming the branch points of $\pi$ and the divisor $D_{\syz}\subseteq \cG_{\EE}$ is locally given by $(t_1=0)$. Therefore
$\cG_{\EE}$ is not normal along $D_{\mathrm{syz}}$.

\begin{notation} We denote by $\widetilde{\mathcal{G}}_{E_6}\rightarrow \cG_{\EE}$ the normalization map.
Let
$$\tilde{f}\colon \widetilde{\mathcal{C}}_{E_6}\rightarrow \widetilde{\mathcal{G}}_{E_6}$$ be the universal curve over
$\widetilde{\mathcal{G}}_{E_6}$.

Finally, we denote by $\widetilde{\beta}\colon \widetilde{\hur} \dashrightarrow \widetilde{\cG}_{\EE}$ the map induced from $\beta$ by the universal property of the normalization $\widetilde{\cG}_{\EE}\rightarrow \cG_{\EE}$. We still denote by $D_0$, $D_{\mathrm{syz}}$ and $D_{\mathrm{azy}}$ the reduced boundary divisors on $\widetilde{\cG}_{\EE}$ corresponding to the same symbols under the map $\widetilde{\beta}$, that is, $\widetilde{\beta}^*(D_{0})=D_{0}$, $\widetilde{\beta}^*(D_{\mathrm{syz}})=D_{\mathrm{syz}}$ and $\widetilde{\beta}^*(D_{\mathrm{azy}})=D_{\mathrm{azy}}$.
\end{notation}

Along the divisor $D_{\mathrm{syz}}$, the space  $\widetilde{\mathcal{G}}_{E_6}$ consists of \emph{two} sheets having local coordinates $(s,t_2, \ldots, t_{21})$, such that the map $\widetilde{\mathcal{G}}_{E_6}\rightarrow \cG_{\EE}$ is given locally by $$(u=s, v=s, t_1=s^2) \ \mbox{ and } \ (u=-s,v=s, t_1=s^2)$$
respectively. Accordingly, the fibre product  $\mathcal{C}'_{\EE}:=\mathcal{C}_{E_6}\times_{\cG_{\EE}}\widetilde{\mathcal{G}}_{E_6}$ has $A_1$-singularities along the codimension $2$ locus  corresponding to nodes $\bigl([C \rightarrow R],z\in C_{\mathrm{sing}}\bigr)$ over points in $D_{\syz}$. Indeed, if $xy=t_1$ is the local equation of $\mathcal{C}_{E_6}$ in coordinates $(x,y,t_1, \ldots, t_{21})$, then the local equation of $\mathcal{C}'_{\EE}$ is $xy=s^2$.  Observe that $\widetilde{\mathcal{C}}_{E_6}$ is obtained from $\mathcal{C}'_{\EE}$ by blowing-up the locus of nodes. It follows that
over a point $[\oC \rightarrow R]\in D_{\syz}$, we have $$\tilde{f}^{-1}\bigl([\oC \rightarrow R]\bigr)=C_1\cup_{\{p_1, p_2\}} E,$$ where $E$ is a smooth rational curve meeting the smooth curve $C_1$ at $p_1$ and $p_2$.

\begin{notation}
We denote by $\mathcal{L}$ a universal line bundle over $\widetilde{\mathcal{C}}_{E_6}$. For a point $[\oC=C_1\cup E, \oL]\in D_{\mathrm{syz}}$ as above, we have  $\mathcal{L}_{|C_1}=\nu^*(\oL)\in W^1_{27}(C_1)$ and $\mathcal{L}_{|E}=\mathcal{O}_E$.
\end{notation}

\begin{theorem}\label{lamge6}
At the level of $\widetilde{\mathcal{G}}_{E_6}$ one has the following formula:
$$\lambda=\frac{33}{46}[D_0]+\frac{7}{46} [D_{\mathrm{azy}}]+\frac{17}{46} [D_{\mathrm{syz}}] \in CH^1(\widetilde{\mathcal{G}}_{E_6}).$$
\end{theorem}
\begin{proof}
We study the map $\varphi:=\widetilde{\beta}\circ q \colon \ocH\dashrightarrow \widetilde{\mathcal{G}}_{E_6}$. At the level of $\ocH$ we have the formula \cite[Theorem 6.17]{alexeev15the-uniformization}:
$$\lambda=\frac{7}{23}[E_{\mathrm{azy}}]+\frac{17}{46}[E_{\syz}]+\frac{33}{28}[E_0]+\cdots \in CH^1(\ocH).$$
We claim that $\varphi^*([D_0])=2[E_0]$, $\varphi^*([D_{\azy}])=2[E_{\azy}]$ and $\varphi^*([D_{\syz}])=[E_{\syz}]$ which explains the result.

\vskip 3pt

We start with a family of $W(E_6)$-pencils
$\bigl(f_t\colon C_t\rightarrow \mathbb P^1\bigr)_{t\in T}$ and assume that over a special point $t_0\in T$, two branch points coalesce. Depending on the situation,
the curve $C_0$ is smooth (in the azygetic case), or nodal (in the syzygetic, or the $D_0$-case). In order to separate the branch points one makes a base change
of order $2$ which justifies the multiplicity in front of both $E_0$ and $E_{\azy}$. This base change is not needed in the case $E_{\syz}$ for, when we passed to the normalization, the two branches were separated.
\end{proof}

\begin{remark}
Observe that a formula identical to Theorem \ref{lamge6} has been established in \cite[Remark 5.21]{alexeev15the-uniformization}  at the level of $\widetilde{\hur}$. The stacks
$\widetilde{\hur}$ and $\widetilde{\mathcal{G}}_{E_6}$ are however not isomorphic over the divisors $D_0, D_{\azy}$ and $D_{\syz}$. For instance, over a general point in $D_{\azy}$
the non-normalized Harris-Mumford space $\mathcal{HM}_{E_6}$ of admissible covers has local equations
$$s_1^3=\cdots=s_6^3=t_1,$$
in local coordinates $(s_1,\ldots, s_6,t_1,\ldots, t_{21})$, where $D_{\azy}$ is given by $(t_1=0)$.
Accordingly, the local equation of $\widetilde{\hur}$ (which locally is the normalization of $\mathcal{HM}_{E_6}$) in coordinates $(a, t_2, \ldots, t_{21})$ is given by
$s_1=\zeta_1 a, \ldots, s_6=\zeta_6 a,\  t_1=a^3$,
where $\zeta_1, \ldots, \zeta_6$ are primitive cubic roots of unity and $a$ is a local parameter. In particular, over a general point of $D_{\azy}$ in $\widetilde{\mathcal{G}}_{E_6}$ there lie $3^5=\frac{1}{3}\times 3^6$ points in $\widetilde{\hur}$.
\end{remark}

\begin{theorem}\label{kappae}
We have the following formula:
$$\kappa=12\lambda-6[D_0]-2[D_{\syz}]\in CH^1(\widetilde{\mathcal{G}}_{E_6}).$$
\end{theorem}
\begin{proof}
By definition $\kappa=\tilde{f}_*\bigl(c_1^2(\omega_{\tilde{f}})\bigr)$. We apply Grothendieck-Riemann-Roch to the universal curve $\tilde{f}\colon \widetilde{\mathcal{C}}_{E_6}\rightarrow \widetilde{\mathcal{G}}_{E_6}$. The usual calculation of Mumford yields
$$\kappa=12\lambda-\tilde{f}_*[\mathrm{Sing}(\tilde{f})].$$
The general point of $D_0$ has $6$ singularities, thus explaining the factor $6[D_0]$. Similarly, the general point of $D_{\syz}$ corresponds to a curve with \emph{two} singularities, namely the points of intersection $E\cap C_1$, keeping the notation above. This explains the factor $2[D_{\syz}]$.
\end{proof}

\subsection{Tautological classes on $\widetilde{\cG}_{\EE}$} In \cite[9.6]{alexeev15the-uniformization}, after having chosen a universal line bundle $\mathcal{L}$ on the universal curve $\widetilde{\cC}_{\EE}$, the following tautological classes over $\widetilde{\cG}_{\EE}$ were defined:

$$\mathfrak{A}:=\tilde{f}_*\bigl(c_1^2(\mathcal{L})\bigr), \mbox{ } \mbox{  }  \mathfrak{B}:=\tilde{f}_*\bigl(c_1(\mathcal{L})\cdot c_1(\omega_{\tilde{f}})\bigr), \ \gamma:=\mathfrak{B}-\frac{5}{3}\mathfrak{A}\in CH^1(\widetilde{\mathcal{G}}_{E_6}).$$
Whereas $\mathfrak{A}$ and $\mathfrak{B}$ depend on the choice of a universal line bundle $\mathcal{L}$ on $\widetilde{\cC}_{\EE}$, the class $\gamma$ is intrinsically defined and does not depend on such a choice. We define the \emph{tautological part} of $CH^1(\widetilde{\mathcal{G}}_{E_6})$ to be the three dimensional subspace with the following three distinguished bases:

\vskip 3pt

\begin{itemize}
\item
$(D_{\mathrm{azy}}, D_{\mathrm{syz}}, D_{0})$. All calculations on $\ohur$ are carried out using it.
\item $(\lambda, \gamma, D_{0})$. This basis is best suited for working with the space $\widetilde{\mathcal{G}}_{E_6}$.
\item $(\lambda, \lambda^{(-5)}, D_{0})$. This is the basis compatible with the Prym-Tyurin map $PT$.
\end{itemize}

In what follows we clarify the relation between these bases:

\begin{theorem}\label{azy3} The following relation holds:\footnote{Theorem \ref{azy3} corrects Theorem 8.14 from \cite{alexeev15the-uniformization}, where the non-normality of $\mathcal{G}_{E_6}$ along $D_{\mathrm{syz}}$ was not accounted for.}
$$[D_{\mathrm{azy}}]= \gamma+4\lambda-3[D_0]-2[D_{\mathrm{syz}}]\in CH^1(\widetilde{\mathcal{G}}_{E_6}).$$
\end{theorem}
\begin{proof}
We represent $D_{\mathrm{azy}}$ as the push-forward of the codimension two locus in the universal curve $\widetilde{\mathcal{C}}_{E_6}$ of the locus of pairs $[C \rightarrow R,p]$, where $p\in C$ is such that $h^0(C, L(-3p))\geq 1$. We form the fibre product of the universal curve $\widetilde{\mathcal{C}}_{E_6}$ together with its projections:
$$
\begin{CD}
{\widetilde{\mathcal{C}}_{E_6}} @<\pi_1<< {\widetilde{\mathcal C}_{E_6}\times_{\widetilde{\mathcal{G}}_{E_6}}\widetilde{\mathcal{C}}_{E_6}} @>\pi_2>> {\widetilde{\mathcal{C}}_{E_6}} \\
\end{CD}.
$$
For each $k\geq 1$, we consider the \emph{locally free} jet bundle $J_k(\mathcal{L})$ defined, e.g., in \cite{E96}, as a locally free replacement
(that is, double dual) of the sheaf of principal parts $\mathcal P_{\tilde{f}}^k(\cL):=(\pi_2)_{*}\Bigl(\pi_1^*(\mathcal{L})\otimes \mathcal{I}_{(k+1)\Delta}\Bigr)$ on $\widetilde{\mathcal{C}}_{E_6}$.
Note that $\mathcal P_{\widetilde{f}}^k(\mathcal{L})$ is not locally free along the codimension two locus in $\widetilde{\mathcal{C}}_{E_6}$ where $\tilde{f}$ is not smooth. To remedy this problem,  we consider the \emph{wronskian} locally free replacements $J_{\tilde f}^k(\cL)$, which are related by the following commutative diagram for each $k\geq 1$:
$$
     \xymatrix{
         0 \ar[r] & \Omega_{\tilde{f}}^{k}\otimes \cL \ar[r]^{} \ar[d]_{} & \mathcal{P}_{\tilde{f}}^k(\cL) \ar[d]_{} \ar[r]^{} & \mathcal{P}_{\tilde{f}}^{k-1}(\cL) \ar[d] \ar[r] & 0 \\
          0 \ar[r] & \omega_{\tilde{f}}^{\otimes k}\otimes \mathcal{L} \ar[r] &J_{\tilde{f}}^k(\mathcal{L}) \ar[r]      & J_{\tilde{f}}^{k-1}(\mathcal{L}) \ar[r] & 0. }
$$
Here $\Omega_{\tilde{f}}^k$ denotes the $\mathcal O_{\widetilde{\mathcal G}_{E_6}}$-module $\mathcal{I}_{k\Delta}/\mathcal{I}_{(k+1)\Delta}$. The first vertical row here is induced by the canonical map $\Omega_{\tilde{f}}^k\rightarrow \omega_{\tilde{f}}^{\otimes k}$ relating the sheaf of relative K\"ahler differentials to the relative dualizing sheaf of the family $\tilde{f}$. The sheaves $\mathcal P^k_{\tilde{f}}(\cL)$ and $J^k_{\tilde{f}}(\cL)$ differ only along the codimension two singular locus of $\tilde{f}$.
Setting $\cV := \tilde{f}_* \cL$, there is, for each integer $k\geq 0$, a vector bundle morphism $\nu_k\colon \tilde{f}^*(\cV)\rightarrow J_{\tilde{f}}^k(\cL)$, which for points $[C,L,p]\in \widetilde{\mathcal{G}}_{E_6}$ such that $p\in C_{\mathrm{reg}}$, is just the evaluation morphism $H^0(C,L)\rightarrow H^0(L |_{ (k+1)p})$. We specialize now to the case $k=2$ and consider the codimension two locus $Z\subseteq \widetilde{\mathcal{C}}_{E_6}$ where $$\nu_2\colon \tilde{f}^*(\cV)\rightarrow J_{\tilde{f}}^2(\mathcal{L})$$ is not injective. Then, at least over the locus of smooth curves, $D_{\mathrm{azy}}$ is the set-theoretic image of $Z$. Furthermore, a local analysis shows that the morphism $\nu_2$ is simply degenerate for each point $[C,L,p]$, where $p\in C_{\mathrm{sing}}$. Taking into account that a general point of $D_{\mathrm{azy}}$ corresponds to a pencil with \emph{six} triple points aligned over one branch point, and that the stable model of a general element of the divisor $D_{\mathrm{syz}}$ corresponds to a curve with \emph{one} node, whereas that of a general point of $D_{0}$ to a curve with \emph{six} nodes, we obtain the formula:
\[
6[D_{\mathrm{azy}}]= \tilde{f}_* c_2\left(\frac{J_{\tilde{f}}^2(\mathcal{L})}{\tilde{f}^*(\cV)}\right)-6[D_0]-8[D_{\mathrm{syz}}]\in CH^1(\widetilde{\mathcal{G}}_{E_6}).
\]
The fact that $D_{\mathrm{syz}}$ appears with multiplicity $8$ is a result of the fact that  $\tilde{f}^{-1}([C,L])=\widetilde{C}\cup_{\{p_1, p_2\}} E$, over a general point $[C,L]\in D_{\syz}$ has \emph{two} singularities, and that, at each of the nodes, there is a local multiplicity equal to $4$ as we shall explain.

\vskip 4pt

We choose a family $F\colon \mathcal{X} \rightarrow B$ of curves of genus $46$ over a smooth $1$-dimensional base $B$, such that $\mathcal{X}$ is smooth, and there is a point $b_0\in B$ such that $X_b:=F^{-1}(b)$ is smooth for $b\in B\setminus \{b_0\}$, whereas $X_{b_0}$ has precisely two nodes $p_1$ and $p_2$. Assume $L\in \mbox{Pic}(\mathcal{X})$ is a line bundle such that $L_b:=L_{|X_b}$ is a pencil with $\WE$-monodromy on $X_b$ for each $b\in B$, and furthermore $[X_{b_0}, L_{b_0}]\in D_{\mathrm{syz}}$. We have that $X_{b_0}=C\cup_{\{p_1, p_2\}} E$, where $C$ is a smooth curve of genus $45$ and $E$ is a smooth rational curve, meeting $C$ at the nodes $p_1$ and $p_2$.

\vskip 3pt

Choose local parameters $t\in \mathcal{O}_{B, b_0}$ and $u, v\in \mathcal{O}_{\mathcal{X}, p_1}$, such that $uv=t$ represents the local equation of $\mathcal{X}$ around the point $p_1$. Here $u$ is the local parameter on $C$, whereas $v$ is the local parameter on $E$. Then $\omega_F$ is locally generated at the point $p_1\in \mathcal{X}$ by the meromorphic differential $\tau=\frac{du}{u}=-\frac{dv}{v}$. We choose two sections $s_1, s_2\in H^0(\mathcal{X}, L)$, where $s_1$ does not vanish at $p_1$ or $p_2$ and $s_2$ vanishes with order $2$ at $p_1, p_2$ along $C$, while being identically zero along $E$. Thus (after a local analytic change of coordinates) we can write a relation $s_{2, p_1}=u^2s_{1, p_1}$ between the germs of the two sections $s_1$ and $s_2$ at $p_1$. We compute
$$d(s_2)-2udu=d(s_2)-2u^2\tau\in (u,v)\tau, \ \ \mbox{ and } \ \ d^2(s_2)-4udu=d^2(s_2)-4u^2\tau\in (u,v)\tau.$$
In local coordinates, the map $H^0\bigl(X_{b_0}, L_{b_0}\bigr) \rightarrow H^0\bigl(X_{b_0}, L_{b_0}|_{3p_1}\bigr)$ is then given by the following matrix,
$$\begin{pmatrix}
1 & 0& 0\\
u^2 & 2u^2+(u,v) & 4u^2+(u,v)\\
\end{pmatrix},
$$
where the symbol $f+(u,v)$, indicates an element of $\mathcal{O}_{\mathcal{X}, p_1}$ that differs from $f$ by an element in the ideal $(u,v)$. The local equations of the degeneracy locus $Z$ are the two by two minors of the above matrix. This shows that the local multiplicity coming from the node $p_1\in X_{b_0}$ of $[D_{\mathrm{syz}}]$ in $Z$ is equal to $4$, hence $[D_{\syz}]$ appears with multiplicity $8=4+4$ in the degeneracy locus.\footnote{In \cite[Theorem 9.12]{alexeev15the-uniformization} there is a mistake in a similar calculation: the multiplicity there is $4$ and not $3$.}

\vskip 4pt

We compute: $c_1\bigl(J_{\tilde{f}}^2(\cL)\bigr)=3c_1(\cL)+3c_1(\omega_{\tilde{f}})$ and $c_2\bigl(J_{\tilde{f}}^2(\cL)\bigr)=3c_1^2(\cL)+6c_1(\cL)\cdot
c_1(\omega_{\tilde f})+2c_1^2(\omega_{\tilde f})$,
hence
$$\tilde{f}_* c_2\left(\frac{J_{\tilde{f}}^2(\mathcal{L})}{\tilde{f}^*(\cV)}\right)=3\mathfrak{A}+6\mathfrak{B}-3(d+2g-2)c_1(\cV)+2\kappa=6\gamma+2\kappa.$$
As explained in Theorem \ref{kappae}, we also have $\kappa=12\lambda-6[D_{E_6}]-2[D_{\mathrm{syz}}]$, which finishes the proof.
\end{proof}


Recall that $\tilde{f}\colon \widetilde{\mathcal{C}}_{E_6}\rightarrow \widetilde{\mathcal{G}}_{E_6}$ denotes the universal curve and $\mathcal{L}$ is a universal line bundle of relative degree $27$ over $\widetilde{\mathcal{C}}_{E_6}$. The push-forward sheaves $\tilde{f}_*(\mathcal{L})$ and $\tilde{f}_*\bigl(\omega_{\tilde{f}}\otimes \mathcal{L}^{\vee}\bigr)$ are reflexive sheaves, therefore using  \cite{Ha}, both are locally free outside a subset of codimension at least $3$ in $\widetilde{\mathcal{G}}_{E_6}$. By possibly removing this locus, for all divisor class calculations that follow, we may assume that both $\tilde{f}_*(\mathcal{L})$ and $\tilde{f}_*\bigl(\omega_{\tilde{f}}\otimes \mathcal{L}^{\vee}\bigr)$ are locally free. Using \cite[Lemma 11.5]{alexeev15the-uniformization}, for a general point $[\pi\colon C\rightarrow \mathbb P^1]\in
\widetilde{\mathcal{G}}_{E_6}$, if $L:=\pi^*(\mathcal{O}_{\mathbb P^1}(1))$,  we have $h^0(C,L)=2$ and $h^0(C, \omega_C\otimes L^{\vee})=20$, therefore by Grauert's Theorem
$$\mbox{rk}\bigl(\tilde{f}_*(\mathcal{L})\bigr)=2 \ \mbox{ and } \ \mbox{rk}\big(\tilde{f}_*\bigl(\omega_{\tilde{f}}\otimes \mathcal{L}^{\vee}\bigr)\bigr)=20.$$

\vskip 3pt

We fix a point $[\pi\colon C\rra \mathbb P^1]=[C,L]\in \tcG_{E_6}$ and a point $p\in \mathbb P^1$ such that $\pi^{-1}(p)\subseteq C_{\mathrm{reg}}$. We consider the usual cohomology exact sequence on $C$

\begin{equation}\label{exseq1}
0 \lra H^0 (C, \cO_C) \lra H^0(C,L) \lra H^0(\cO_{\Gamma_p}(\Gamma_p)) \stackrel{\alpha_p}{\lra} H^1(C, \cO_C) \lra H^1(C,L) \lra 0,
\end{equation}
where $\Gamma_p$ is the divisor of $|L|=|\pi^*\cO_{\mathbb P^1}(1)|$ above $p$. We identify $H^0(\cO_{\Gamma_p}(\Gamma_p))$ with the $\mathbb C$-vector space spanned by the $27$ lines on a fixed cubic surface $S$. The incidence correspondence on the set of lines of $S$ induces an endomorphism
$$\gamma_p\colon H^0(\cO_{\Gamma_p}(\Gamma_p))\rightarrow H^0(\cO_{\Gamma_p}(\Gamma_p))$$
with eigenvalues $10, 1$ and $-5$, with eigenspaces $H^0(\cO_{\Gamma_p}(\Gamma_p))^{(10)}$, $H^0(\cO_{\Gamma_p}(\Gamma_p))^{(1)}$
and $H^0(\cO_{\Gamma_p}(\Gamma_p))^{(-5)}$ of dimensions $1$, $20$ and $6$ respectively. Note that $H^0(\cO_{\Gamma_p})^{(+10)}$ is spanned by the sum of all the $27$ lines on $S$ and, as in the proof of \cite[Theorem 9.3]{alexeev15the-uniformization}, the space $H^0(\cO_{\Gamma_p}(\Gamma_p))^{(+10)}$ can be identified with the trivial representation of $W(E_6)$. Furthermore, if $D\colon H^0(C, \omega_C)\rightarrow H^0(C,\omega_C)$ is the endomorphism induced by the Kanev correspondence on $C$, the following diagram is commutative for each $p\in \mathbb P^1$:
$$
     \xymatrix{
           H^0(\cO_{\Gamma_p}(\Gamma_p)) \ar[d]_{\gamma_p} \ar[r]^{\alpha_p} & H^0(C,\omega_C)^{\vee} \ar[d]_{D^{\vee}}  \\
           H^0(\cO_{\Gamma_p}(\Gamma_p)) \ar[r]^{\alpha_p}      & H^0(C,\omega_C)^{\vee}. }
$$

Therefore, the decomposition into eigenspaces produces the exact sequences
\[
0 \lra H^0 (C, \cO_C) \lra H^0(C,L)^{(+10)} \lra H^0(\cO_{\Gamma_p}(\Gamma_p))^{(+10)} \lra 0,
\]

and
\begin{equation}\label{seq5}
0 \lra H^0(C,L)^{(-5)} \lra H^0(\cO_{\Gamma_p}(\Gamma_p))^{(-5)} \stackrel{\alpha_p^{(-5)}}{\lra} H^1(C, \cO_C)^{(-5)} \lra H^1(C,L)^{(-5)} \lra 0.
\end{equation}



It follows from \cite[Section 11]{alexeev15the-uniformization} that $h^0(C, L) =2$ (hence $h^1(C, L) =20$) for
a general $[C, L]\in \widetilde{\cG}_{\EE}$, therefore in this case we also have
$H^0 (C, L) = H^0 (C, L)^{(+10)}$ and $H^0(C,L)^{(-5)}=0$ and $H^1 (C, L) = H^1 (C, L)^{(+1)}$.
It also follows that the space $H^0(C,L)^{(+10)}$ can be canonically identified with the subspace $\pi^*H^0(\mathbb P^1, \cO_{\mathbb P^1}(1))$ of $H^0(C,L)$ and it always has dimension $2$.

\subsection{The divisor $\fM$.}\label{fm} The locus of those triples $[C,L,p]\in \widetilde{\cC}_{E_6}$ such that the map
$$\alpha_p^{(-5)}\colon H^0(\cO_{\Gamma_p}(\Gamma_p))^{(-5)}\longrightarrow \bigl(H^0(C,\omega_C)^{\vee}\bigr)^{(-5)}$$
is not an isomorphism can be represented as the pullback $\tilde{f}^*(\fM)$ of an effective divisor $\fM$ on  $\widetilde{\mathcal{G}}_{E_6}$, for the degeneracy of the map $\alpha_p^{(-5)}$ is independent of the choice of a point $p\in \mathbb P^1$.

In what follows we characterize this divisor set-theoretically and observe that, surprisingly, the locus in $\widetilde{\mathcal{G}}_{E_6}$ of pairs $[C,L]$ such that $h^0(C,L)>2$ is of codimension one.

\begin{proposition}\label{prop:M}
 If $[C,L]\in \fM$, then $h^0(C,L)\geq 3$. Furthermore, if $[C,L]\in\widetilde{\mathcal{G}}_{E_6}\setminus \fM$, then
$$\mathrm{Im}\bigl\{ H^0(C,L)\otimes H^0(C,\omega_C\otimes L^{\vee})\rightarrow H^0(C,\omega_C)\bigr\}\subseteq H^0(C, \omega_C)^{(+1)}.$$
\end{proposition}
\begin{proof}
Assume $h^0(C,L)=2$, therefore $H^0(C,L)=H^0(C,L)^{(+10)}$. From the sequence (\ref{seq5}), it follows that $\alpha_p^{(-5)}$ is injective, hence by comparing dimensions, it is an isomorphism, that is, $[C,L]\notin \fM$.

In order to establish the second claim, we use the exactness of the second half of the sequence (\ref{seq5}). Since $\mbox{Im}\bigl(\alpha_p^{(-5)}\bigr)=H^0(C,\omega_C)^{(-5)}$, in particular $\mbox{Im}(\alpha_p)\supseteq \bigl(H^0(C, \omega_C)^{\vee}\bigr)^{(-5)}$. By dualising, if $s\in H^0(C,L)$ is the section defining the divisor $\Gamma_p$, we obtain that $s\cdot H^0(C,\omega_C\otimes L^{\vee})\subseteq H^0(C,\omega_C)^{(+1)}$, which establishes the claim, by varying the section $s\in H^0(C,L)$.
\end{proof}

\vskip 5pt

\subsection{The divisor $\fN$.}\label{fn}

We define  the \emph{Weyl-Petri divisor} $\fN$ to be degeneracy locus of the map of vector bundles of rank $40$
$$\mu\colon \tilde{f}_*(\mathcal{L})\otimes \tilde{f}_*(\omega_{\tilde{f}}\otimes \mathcal{L}^{\vee})\rightarrow \tilde{f}_*(\omega_{\tilde{f}})^{(+1)}$$
over $\widetilde{\mathcal{G}}_{E_6}$. Observe that away from the divisor $\fM$, the points in $\fN$ are precisely those for which the
Petri map $\mu(L)\colon H^0(C,L)\otimes H^0(C, \omega_C\otimes L^{\vee})\rightarrow H^0(C,\omega_C)$ is not injective.

\begin{lemma}
For each point $[\pi\colon C\rightarrow \mathbb P^1]\in \widetilde{\mathcal{G}}_{E_6}$, one has the identification
$\tilde{f}_*(\mathcal{L})[\pi]\cong H^0(C,L)^{(+10)}.$
\end{lemma}

\begin{proof}
Use that $\tilde{f}_*(\mathcal{L})$ is locally free, coupled with the sequence (\ref{exseq1}).
\end{proof}

In what follows we shall determine the class of the divisor $\fN$.

\begin{proposition}\label{che}
The following formula holds at the level of $\widetilde{\mathcal{G}}_{E_6}$:
$$[\mathfrak{N}]=\lambda^{(+1)}-2\lambda+\gamma=\lambda^{(-5)}=-\lambda-\lambda^{(-5)}+\gamma.$$
\end{proposition}
\begin{proof} Using the description of $\mathfrak{N}$ as a degeneracy locus, we compute  that
$$[\mathfrak{N}]=\lambda^{(+1)}-c_1\bigl(\tilde{f}_*(\mathcal{L})\otimes \tilde{f}_*(\omega_{\tilde{f}}\otimes \mathcal{L}^{\vee})\bigr)=\lambda^{(+1)}+c_1\bigl(\tilde{f}_*(\mathcal{L})\otimes R^1 \tilde{f}_*(\mathcal{L})\bigr).$$
Using \cite[Proposition 9.11]{alexeev15the-uniformization}, we have $\mathfrak{A}=27c_1(\tilde{f}_*(\mathcal{L}))$.
Applying Grothendieck-Riemann-Roch to the universal curve $\tilde{f}\colon \widetilde{\cC}_{\EE}\rightarrow \widetilde{\cG}_{\EE}$, we write
$$c_1\bigl(\tilde{f}_*(\mathcal{L})\bigr)-c_1\bigl(R^1\tilde{f}_*(\mathcal{L})\bigr)=\tilde{f}_*\Bigl[\frac{c_1^2(\mathcal{L})}{2}-\frac{c_1(\mathcal{L})
\cdot c_1(\omega_{\tilde{f}})}{2}+\frac{1}{12}\bigl(c_1^2(\omega_{\tilde{f}})-[\mathrm{Sing}(\tilde{f})]\bigr)\Bigr]=\frac{\mathfrak{A}}{2}-\frac{\mathfrak{B}}{2}+\lambda,$$
which leads to the claimed formulas.
\end{proof}

Combining Theorem \ref{azy3} and Proposition \ref{che}, we obtain the following relation:

\begin{theorem}
In the $(\lambda, D_{\syz},D_0)$ basis of $CH^1(\widetilde{\mathcal{G}}_{E_6})$, we have:
$$[\mathfrak{N}]=\frac{59}{42}\lambda-\frac{12}{7}[D_0]-\frac{29}{84}[D_{\syz}],$$
and
$$\gamma=\frac{18}{7}\lambda-\frac{3}{7}[D_{\syz}]-\frac{12}{7}[D_0].$$
\end{theorem}
\begin{proof} Put together Theorem, \ref{azy3}, Proposition \ref{che}, together with the relation $\lambdam=\frac{1}{6}\lambda-\frac{1}{12}[D_{\syz}]$.
\end{proof}

\vskip 3pt

\begin{remark}
In the $(\lambda, \lambda^{(-5)}, [D_0])$-basis of the tautological part of $CH^1(\widetilde{\mathcal{G}}_{E_6})$, the previous formula can be written as
$$[\mathfrak{N}]=\frac{5}{7}\lambda-\frac{12}{7}[D_0]+\frac{29}{7}\lambda^{(-5)}.$$
\end{remark}

\vskip 5pt

\subsection{The ramification divisor of $PT$.}

We now show that the ramification divisor of the Prym-Tyurin map $PT\colon \hur\rightarrow \mathcal{A}_6$ is contained in the union of the divisors $\fM$ and $\fN$. This improves on our \cite[Theorem 0.3]{alexeev15the-uniformization}. Recall that each $W(E_6)$-cover $[\pi\colon C\rightarrow \mathbb P^1, p_1+\cdots+p_{24}]\in \hur$ induces an \emph{Prym-Tyurin canonical map}
$$\varphi_{(-5)}=\varphi_{|H^0(C,\omega_C)^{(-5)}|} \colon C \rightarrow \mathbb P^5.$$

\begin{theorem}\label{h03}
If the Prym-Tyurin canonical image of a smooth curve  $[C,L]\in\hur$ is contained in a quadric, then $[C,L]\in \fM$, in particular, $h^0 (C, L) \geq 3$.
\end{theorem}

\begin{proof}
Let $Q\subseteq \mathbb P^5$ be a quadric containing the Prym-Tyurin canonical image of $C$. Recall from \cite[Section 10]{alexeev15the-uniformization} that, for each branch point $p_i$ of the map $\pi \colon C\rra \bP^1$, the ramification points $r_{i1}, \ldots, r_{i6}$ have the same image, say $\op_i\in \bP^5$ in the Prym-Tyurin canonical space $\bP^5\cong \bP\bigl(H^0 (C,\omega_{C})^{(-5)}\bigr)^{\vee}$.

Since the Prym-Tyurin canonical image $\varphi_{(-5)}(C)$ is non-degenerate, the quadric $Q$ has rank at least $3$, hence its singular locus is a linear subspace of $\bP^5$ of codimension at least $3$. In particular, $Q$ can be singular  at most $14$ of the points $\op_i$: indeed, if for instance $Q$ is singular at $\op_1, \ldots, \op_{15}$, this implies
$$h^0\Bigl(C, \omega_C\bigl(- \sum_{i=1}^{15}\sum_{j=1}^6 r_{ij}\bigr)\Bigr)\geq 3,$$ which is not possible because
$\omega_C\bigl(-\sum_{1\leq i\leq 15} (r_{i1}+\cdots+r_{i6})\bigr)$ has degree $0$.

\vskip 5pt

Therefore, there exists a branch point $p$ of $\pi$, such that $Q$ is smooth at the image $\op$ of the six ramification points on $\pi$ lying over $p$. Let $\Gamma_p := 2(r_1 + \cdots +r_6) + q_1 + \cdots + q_{15}$ be the divisor of $|L|$ above $p$. We write $H^0(C, \omega_C)^{(-5)}=\langle \eta_0, \eta_1, \ldots, \eta_5 \rangle$, where
$\langle \eta_1, \ldots, \eta_5\rangle =H^0\bigl(C,\omega_C\bigr)^{(-5)}(-r_1-\cdots-r_6)$, therefore $\mbox{ord}_{r_i}(\eta_0)=0$.  Assume the equation defining $Q$ is given by
$$q=a\cdot \eta_0^2+\eta_0\cdot (a_1\eta_1+\cdots+a_5\eta_5)+q_1(\eta_1, \ldots, \eta_5)\in \mbox{Sym}^2 H^0(C, \omega_C)^{(-5)},$$
where $a\in \mathbb C$.
Evaluating $q$ at  $r_i$, we obtain $a=0$. Then $\eta:=a_1\eta_1+\cdots+a_5\eta_5\in H^0(C,\omega_C)^{(-5)}$ satisfies $\mbox{ord}_{r_i}(\eta)\geq 2$, for $i=1,\ldots, 6$. Furthermore, $\eta\neq 0$, because $\op\in Q_{\mathrm{reg}}$, that is, hence
$$\eta \in H^0\bigl(C,\omega_C\bigr)^{(-5)}(-2r_1-\cdots-2r_6)\neq 0.$$
 Note that $\eta$ is the equation of the tangent hyperplane to $Q$ at the point $\op$.

Assume now that $[C,L]\in \widetilde{\mathcal{G}}_{E_6}\setminus (\fM\cup \fN)$, thus the map $\alpha_p^{(-5)}$ is an isomorphism. The dual map can be identified with the evaluation map
$$\bigl(\alpha_p^{(-5)}\bigr)^{\vee}\colon H^0(C,\omega_C)^{(-5)}\rightarrow H^0\bigl(\omega_{C|\Gamma_p}\bigr)^{(-5)},$$
hence we obtain that $H^0\bigl(\omega_{C|\Gamma_p}\bigr)^{(-5)}(-2r_1-\cdots-2r_6)\neq 0$.
Identifying $H^0(\omega_{C|\Gamma_p})^{(-5)}$ with the primitive cohomology of a $1$-nodal cubic surface, this fact  implies in fact that $$H^0\bigl(\omega_{C|\Gamma}\bigr)^{(-5)}(-2r_1-\cdots-2r_6-q_1-\cdots-q_{15})\neq 0,$$ which yields $0\neq \eta\in H^0(C,\omega_C)^{(-5)}(-\Gamma_p)$, that is, $\eta \in \mbox{Im}\bigl\{H^0(C,L)\otimes H^0(C,\omega_C\otimes L^{\vee})\rightarrow H^0(C,\omega_C)\bigr\}$. We conclude $\eta\in H^0(C, \omega_C)^{(-5)}\cap H^0(C, \omega_C)^{(+1)}=\{0\}$, which is a contradiction.

\end{proof}

\vskip 4pt

\noindent \emph{Proof of Theorem \ref{thm:ram}.} It suffices to combine Theorem \ref{h03} with \cite[Theorems 0.3 and 9.3]{alexeev15the-uniformization}, asserting that a point $[C,L]\in \widetilde{\mathcal{G}}_{E_6}\setminus \fN$ lies in the ramification divisor of $PT$ if and only the Prym-Tyurin canonical curve $\varphi_{(-5)}(C)$ lies on a quadric.
\hfill $\Box$

\section{A universal theta divisor on the moduli space of $\WE$-covers}
\label{sec:degenerate-curve}

In this section we discuss the geometry of a very natural effective divisor on $\widetilde{\hur}$, which can be viewed as
(a translate of) the universal theta divisor (not to be confused with the pull-back of the universal theta divisor
from $\overline{\mathcal{A}}_6$). Since the geometric construction we are interested in is defined directly in terms of a $\WE$-pencil,
it is easier to work again with the parameter space  $\widetilde{\mathcal{G}}_{E_6}$.

\begin{definition} We consider the following locus inside $\widetilde{\mathcal{G}}_{E_6}$
  \begin{equation}\label{virtdiv1}
    \mathfrak{D}_1:=\Bigl\{[C,L]\in \widetilde{\cG}_{\EE}: H^0\bigl(C,
    2\omega_C-5L\bigr)\neq 0\Bigr\}.
  \end{equation}

Note that since $\mbox{deg}(2\omega_C-5L)=g(C)-1=45$, points in $\mathfrak{D}_1$ are characterized by the
  condition that $2\omega_C-5L$ lies in the theta divisor $W_{45}(C)\subseteq \mbox{Pic}^{45}(C)$. In
  particular, $\mathfrak{D}_1$ is a virtual divisor on
  $\widetilde{\cG}_{\EE}$.
\end{definition}

\begin{theorem}\label{dn}
  The virtual class of $\mathfrak{D}_1$ is given
  by the following formula:
  \begin{displaymath}
    [\mathfrak{D}_1]^{\mathrm{vir}}=-\lambda-\kappa+\frac{15}{2} \gamma\in
    CH^1(\widetilde{\cG}_{\EE}).
  \end{displaymath}
\end{theorem}

\begin{proof}
We reinterpret the defining property of points in $\mathfrak{D}_1$
via the Base Point Free Pencil Trick, as saying that the multiplication map
  \begin{displaymath}
    \mu_1(L)\colon H^0(C,L)\otimes H^0\bigl(C, 2\omega_C-4L\bigr)\longrightarrow  H^0\bigl(C,
    2\omega_C-3L \bigr)
  \end{displaymath}
  is not bijective.  Note that one has $h^0(2\omega_C-4L)=27$
  and that $h^0(C, 2\omega_C-3L)=54$. Furthermore, using the construction given in \ref{subsekt} of the birational isomorphism $\tbeta\colon
\widetilde{\hur} \dashrightarrow \tcG_{\EE}$, it follows that $L$ is a base point free pencil for every point $[C,L]\in \widetilde{\cG}_{\EE}$. The map $\mu_1(L)$ can be
  globalized to a morphism of vector bundles over $\widetilde{\cG}_{\EE}$ having the same rank
  \begin{displaymath}
    \mu_1\colon  \tilde{f}_*(\mathcal L) \otimes \tilde{f}_*\bigl(\omega_{\tilde{f}}^{\otimes
      2}\otimes \mathcal{L}^{\otimes (-4)}\bigr) \
    \longrightarrow \ \tilde{f}_*\bigl(\omega_{\tilde{f}}^{\otimes 2}\otimes
    \mathcal{L}^{\otimes (-3)}\bigr),
  \end{displaymath}
  where, as in the previous section, $\mathcal{L}$ is a universal
  pencil with $\WE$-monodromy over the universal curve
  $\tilde{f}\colon \widetilde{\cC}_{\EE}\rightarrow \widetilde{\cG}_{\EE}$.  Clearly,
  $\mathfrak{D}_1$ is the degeneracy locus of $\mu_1$.

  \vskip 3pt

  Since one has
  \begin{displaymath}
    R^1\tilde{f}_*\Bigl(\omega_{\tilde{f}}^{\otimes 2}\otimes
  \mathcal{L}^{\otimes (-4)}\Bigr)=0, \quad
  R^1{\tilde{f}}_*\Bigl(\omega_{\tilde f}^{\otimes 2}\otimes \mathcal{L}^{\otimes
    (-3)}\Bigr)=0,
  \end{displaymath}
  the Chern classes of the sheaves appearing in the definition of the
  morphism $\mu_1$ can be computed via a Grothendieck-Riemann-Roch calculation.  For
  instance,
  \begin{displaymath}
    c_1\Bigl({\tilde f}_*\bigl(\omega_{\tilde f}^{\otimes 2}\otimes
    \mathcal{L}^{\otimes (-4)}\bigr)\Bigr)=\lambda+\kappa+8\mathfrak{A}-12\mathfrak{B},
  \end{displaymath}
  and after routine manipulations we obtain the claimed formula.
\end{proof}

\begin{corollary}
The (virtual) class of $[\mathfrak{D}_1]$  in the
 $(\lambda, [D_{\syz}], [D_0])$ basis of $\mathrm{Pic}(\widetilde{\mathcal{G}}_{\EE})$ is given by:
  \begin{displaymath}
    [\mathfrak{D}_1]^{\mathrm{virt}}=\frac{44}{7}\lambda-\frac{17}{14}[D_{\syz}]-\frac{48}{7}[D_{0}].
  \end{displaymath}
\end{corollary}

\vskip 4pt

\subsection{A  degenerate $\WE$-cover} It is crucial to establish that the virtual divisor $\mathfrak{D}_1$ is a genuine divisor on $\widetilde{\mathcal{G}}_{E_6}$.
To that end we shall use degeneration and we first need some preparation. We start once more with a $\WE$-cover $[\pi\colon C\to\bP^1, p_1+\cdots+p_{24}]\in \hur.$ Recall that fibers of $\pi$ over a generic point in $\bP^1$ can be
identified with the lines $\ell_1, \dotsc, \ell_{27}$ on a fixed
smooth cubic surface $S$, as well as with the $(-1)$-vectors in the orbit
$\WE.\varpi_6$ of the coweight lattice $\Lambda^*_{\WE}$. The
reflections $w\in \WE$ can be identified with the roots of the root
lattice $\Lambda_{\WE}$ modulo $\pm 1$: the roots $+r$ and $-r$ give the same
reflection. For each root $r$ there are exactly 6 coweights $a_{r, i}$ with
$(r,a_{r,i})=1$ and $6$ coweights $b_{r, i}$ with $(r,b_{r, i})=-1$ so that $b_{r, i}=a_{r, i}+r$. The
switch from $r$ to $-r$ exchanges $a_{r, i}$'s and $b_{r, i}$'s. Under the
monodromy representation $\WE\into S_{27}$ the reflection $w$
is represented by a \emph{double sixer} $(a_{r,1}, b_{r,1})\cdots (a_{r,6}, b_{r,6})$.

\vskip 4pt

The following lemma describes the basic degeneration used to show that $\mathfrak{D}_1$ is a genuine divisor.
This degeneration will also prove to be instrumental in the final step of the proof of Theorem \ref{thm:second-proof}.

\begin{lemma}\label{lem:glued-curve}
  Let $\cC:=\bigl(\pi_t \colon C_t \to \bP^1, p_1(t), \ldots, p_{24}(t)\bigr)$ be a 1-parameter family
  of $\WE$-covers such that the local
  monodromies $w_i$ of the points $p_i$ are pairwise equal: $w_{2i-1} = w_{2i}$ for $i=1,\dotsc, 12$. Assume
  $\lim p_{2i-1}(t)=\lim p_{2i}(t)=q_i\in \mathbb P^1$.
Then the family $\cC$ can be
  flatly completed to a family of covers of $\bP^1$ so that the
  central fiber $C=C_0$ is a nodal curve labeled by the lines
  $\ell_1,\dotsc, \ell_{27}$, a union of $27$ copies of $\bP^1$ each mapping
  isomorphically down to the base $\bP^1$. The sheets are glued as
  follows. For each point $q_j\in\bP^1$, $j=1,\dotsc, 12$ with local
  monodromy $w_j$, glue the point above $q_j$ on the sheet labelled by $a_{jk}$ to
  the point above $q_j$ on the sheet $b_{jk}$, for $k=1,\dotsc,6$.
\end{lemma}
\begin{proof}
For a generic point $t\in \mathbb P^1$, each ramification point over $p_i(t)$ is of the
  form $y^2=x$, with the 6 pairs $(a_{ik},b_{ik})$ coming together. It is
  immediate that when two branch points on the base come together, the
  limit points on $C$ are nodes. Let $\coprod_{s=1}^{m} \wC_s$ be the normalization of $C$. It first follows that all the components of $C$ are rational, since the map $C \rightarrow \bP^1$ induces \'etale maps $\wC_s \rightarrow \bP^1$. The dual graph $\Gamma:=(V(\Gamma), E(\Gamma))$ of $C$ is
  connected since the reflections $w_i$ are chosen so that they generate  $\WE$. For the
  arithmetic genus of $C$ one has
  \begin{displaymath}
    \bigl|E(\Gamma)\bigr| - \bigl|V(\Gamma)\bigr| + 1 + \sum_{s=1}^{m}
    p_a(\wC_s) =  \bigl|E(\Gamma)\bigr| - \bigl|V(\Gamma)\bigr| + 1 = 46.
  \end{displaymath}
  Since there are
  $12\times  6=72$ edges, it follows that the number of vertices,
  that is, that of the irreducible components $C_s$ of $C$ is $27$. Thus, the normalization of $C$ is a disjoint union
  of 27 copies of $\bP^1$'s and the gluing is as described.
\end{proof}

\begin{remark}
  The switch from a root $r$ to $-r$ representing the same reflection
  $w$ changes the orientation of the 6 respective edges in the oriented
  dual graph $\Gamma$.
\end{remark}

The glued curve $C=C_0$ comes with an ample line bundle
$L=\pi^*(\cO_{\bP^1}(1))$. It also comes with a Kanev correspondence
sending a point over $x\in \bP^1$ on the sheet labeled $\ell$ to the
10 points in the same fiber on the sheets labeled $\ell'$ such that
$\ell$ and $\ell'$ intersect on the abstract cubic surface $S$. The induced
endomorphism $D$ on $H^0(C, \omega_C)$ satisfies $(D+5)(D-1)=0$ and the
corresponding eigenspaces have dimension 6 and 40, just as on a smooth
curve.  For more details, see
\cite[Sections 4 and 5]{alexeev15the-uniformization}.

\begin{theorem}\label{thm:glued-curve-comps}
  There exists a choice of reflections $w_1=w_2$, \dots,
  $w_{23}=w_{24}$ generating  $\WE$ and of points $q_1,\dotsc, q_{12}\in\bP^1$ for which
  the central curve $C$ and the cover $\pi\colon C\rightarrow \mathbb P^1$ as described above have the following properties:
  \begin{enumerate}
  \item $h^0(C, L)=2$.
  \item The image of the multiplication map $H^0(C, L)\otimes H^0(C, \omega_C\otimes L^{\vee}) \to H^0(C, \omega_C)$
    has dimension $40$.
  \item $h^0(C, \omega_C^{\otimes2}(-5L)) = 0$.
  \item The $6$-dimensional eigenspace $H^0(C, \omega_C)^{(-5)}$  is
    base point free.
  \item The image of the Prym-Tyurin canonical curve $\varphi_{(-5)}(C)$ in
    $\bP \bigl(H^0(\omega_C)^{(-5)}\bigr)^{\vee}$ does not lie on a
    quadric.
  \end{enumerate}
\end{theorem}
\begin{proof}
The computation is reduced to linear algebra. A line
  bundle on $C$ of multidegree $(d_1,\dotsc, d_{27})$ is identified
  with a sheaf $\coprod_{i=1}^{27} \cO_{\bP^1}(d_i)$ with specified twists
  $c_{q,i,j}$ at the nodes where the sheets labelled by $i$ and $j$ are glued over a
  point $q\in\bP^1$. If
  $q_1,\dotsc, q_{12}\in \bA^1 = \bP^1\setminus\{\infty\}$, then a section
  of this line bundle is identified with a collection of polynomials
  $P_i(t)$ of degrees $d_i$ with the values at the nodes matching up
  to multiplication by the twist $c_{q,i,j}$.

\vskip 3pt

  For the sheaf $L\in W^1_{27}(C)$ we choose the multidegree to be  $(1,\dotsc, 1)$ and the twists
  are all equal to $1$. For $\omega_C$ the corresponding degrees are
  $d_i = \bigl|C_i\cap \overline{C\setminus C_i}|-2$. The restriction $\omega_i$ to $C_i$ of a section of $\omega_C$ can be viewed as
  $$\omega_i = \frac{P_i(t)
    dt}{\prod(t-q_{is})},$$ where
  $P_i(t)$ is a polynomial of degree $d_i$.  Here, $q_{is}$ are the nodes lying on the
  sheet labelled by $i$.  The twist at a node over $q\in\bA^1$ joining the sheets
  $i$ and $j$ is the negative of the ratio of residues:
  \begin{displaymath}
    c_{q,i,j} =
    - \Res_q \frac{dt}{\prod (t-q_{is})} \ / \
    \Res_q \frac{dt}{\prod (t-q_{jt})}.
  \end{displaymath}
  The twists for the line bundles $\omega_C^{\otimes m}(dL)$ are then
  the appropriate products of the above twists.
  We thus reduce the computation of the dimension of the spaces of sections
  $H^0\bigl(C, \omega_C^{\otimes m}(dL)\bigr)$ for any integers $m$ and $d$ to a concrete linear algebra question.

\vskip 4pt

  The eigenspace $H^0(C, \omega_C)^{(-5)}$ is the subspace of
  $H^0(C, \omega_C)$ where for every branch point $q_1,\dotsc,
  q_{12}$ the residues over each of the sheets
  $a_{i1},\dotsc, a_{i6}$ are equal to each other. The subspace
  $H^0(C, \omega_C)^{(+1)}$ is the subspace where the sums of these
  residues are zero.

\vskip 4pt

  We performed the check for a concrete glued curve corresponding to
  the following choices:
  \begin{itemize}
  \item The points $q_i=i\in\bZ \subseteq \bC$.
  \item  The
  following roots, in standard notation for the Minkowski
  space $I^{1,6}$:

  \noindent $\alpha_{135}=e_0-e_1-e_3-e_5$, $\alpha_{12}=e_1-e_2$, $\alpha_{23}=e_2-e_3$, $\alpha_{34}=e_3-e_4$,
  $\alpha_{45}=e_4-e_5$, $\alpha_{56}=e_5-e_6$, $\alpha_{16}=e_1-e_6$, $\alpha_{456}=e_0-e_4-e_5-e_6$,
  $\alpha_{123}=e_0-e_1-e_2-e_3$, $\alpha_{346}=e_0-e_3-e_4-e_6$, $\alpha_{234}=e_0-e_2-e_3-e_4$,
  $\alpha_{156}=e_0-e_1-e_5-e_6$.
  \end{itemize}
  All the computations were done in \emph{Mathematica} and are available at
  \cite{alexeev2018check-slope}.
\end{proof}

As discussed in \cite[Section 11]{alexeev15the-uniformization}, a consequence of parts (1,2,3) of \eqref{thm:glued-curve-comps}
is that the morphism $\mu$ defining the  Weyl-Petri divisor (see \ref{fn}) is generically non-degenerate, that is, $\mathfrak{N}$ is indeed a genuine divisor on $\ohur$. A consequence of the other parts is:

\begin{theorem}\label{thm:2k5}
  For a generic cover $[\pi\colon C\to\bP^1]\in \ohur$, one has
  $H^0(C, 2\omega_C-5L)=0$. Thus $\mathfrak{D}_1$ is a genuine divisor on $\ohur$.
\end{theorem}
\begin{proof}
Indeed, we consider a flat family degenerating to the glued curve as
in Theorem \ref{thm:glued-curve-comps}. In the central fiber the dimension of  $H^0(C, 2\omega_C-5L)$ can
only increase, which the above argument shows not to be the case.
\end{proof}

\section{The Prym-Tyurin map is unramified generically along the divisor $D_0$}\label{sect:dominance}

In this Section we prove Theorem \ref{thm:second-proof} by showing that the differential of the Prym-Tyurin map $PT\colon \ohur\dashrightarrow
\overline{\cA}_6$ is bijective at a general point of the divisor $D_0$ of $\ohur$. We fix throughout the section a suitably general
$\WE$-admissible cover
\begin{equation}\label{bdpoint}
[\pi\colon C= C_1 \cup C_2 \rightarrow R:= R_1\cup_q R_2, p_1+\cdots+p_{24}]\in D_0\subseteq \ohur.
\end{equation}

We shall assume that $C_1$ is a smooth curve of genus $40$. The curve $C_2$ has $21$ components, all rational, with $6$ components mapping to $R_2$ with degree $2$ and the other $15$ mapping isomorphically to $R_2$. The degree $27$ map $\pi_1=\pi_{|C_1}\colon C_1\rightarrow R_1$ has monodromy $\WE$ and is branched precisely at the points $p_1, \ldots, p_{22}\in R_1\setminus \{q\}$.

\begin{definition}
Let $\hur_1$ denote the Hurwitz space of $\WE$-covers  $[\pi_1\colon C_1\rightarrow \bP^1, p_1+\cdots+p_{22}]$ of degree $27$ with branch points at $p_1, \ldots, p_{22}$. The source $C_1$ is a smooth curve of genus $40$ and the local monodromy of $\pi_1$ at each branch point $p_i\in \mathbb P^1$ is given by a reflection in a root of $E_6$. As in the case of covers with $24$ branch points, the curve $C_1$ has a Kanev correspondence which we denote by $D_1$ and which induces an endomorphism $D_1\colon JC_1\rightarrow JC_1$ and a $5$-dimensional Prym-Tyurin variety $PT(C_1,D_1):=\mbox{Im}(D_1-1)\subseteq JC_1$. Put $L_1:=\pi_1^{*}(\mathcal{O}_{\mathbb P^1}(1))\in W^1_{27}(C_1)$.
\end{definition}

\vskip 3pt

Let $\rho \colon C \rra \oC$ be the map contracting $C_2$.  The curve $\oC$ is the stabilization of $C$ and it has $6$ ordinary double points obtained by identifying two points of $C_1$ if they are connected by a component of $C_2$. We denote by $\overline{L}\in W^1_{27}(\overline{C})$ the line bundle characterized by the property $\bigl(\rho_{|C_1}^*(\overline{L})\cong L_1$.

Given a reduced fiber $\Gamma$ of the map $\pi_1\colon C_1\rightarrow \bP^1$, we consider the usual exact sequence, see also (\ref{exseq1})
\begin{equation}\label{cohstandard}
0 \lra H^0 (C_1, \cO_{C_1}) \lra H^0(C_1,L_1) \lra H^0(\cO_{\Gamma}(\Gamma)) \stackrel{\alpha_1}{\lra} H^1(C_1, \cO_{C_1}) \lra H^1(C_1,L_1) \lra 0.
\end{equation}

The map $\alpha_1$ is equivariant for the action of the Kanev correspondence $D_1$, hence it maps the 6-dimensional $(-5)$-eigenspace of $H^0(\cO_{\Gamma}(\Gamma))^{(-5)}$ into the 5-dimensional space $H^1(C, \cO_C)^{(-5)}$. It follows that $h^0 (C_1, L_1) \geq 3$, in particular
$[C_1]\in \mathcal{M}_{40}$ is a Brill-Noether special curve.

\begin{notation}
Let $\fM_1\subseteq \hur_1$ denote the locus where $H^0 (C_1 , \omega_{C_1} \otimes L_1^{\otimes (-2)})\neq 0$.
\end{notation}

We denote by $PT_5 \colon \hur_1 \rra \cA_5$ the Prym-Tyurin map. The proof in \cite[Section 10]{alexeev15the-uniformization} carries through without changes to the case of $22$ branch points so that we have the following result:

\begin{theorem}\label{thm:hur1dom}
The Prym-Tyurin map $PT_5$ is ramified at a point $[\pi_1 \colon C_1 \rra \mathbb P^1, p_1+\cdots+p_{22}]\in \hur_1\setminus \mathfrak{M}_1$ if and only if the Prym-Tyurin canonical image of $C_1$ is contained in a quadric.
\end{theorem}

\subsection{The map $PT_5$ is dominant.} This follows for instance, from the fact that the ordinary Prym map $P\colon \cR_6 \rra \cA_5$ is dominant, using the fact that $6$-dimensional Prym-Tyurin varieties degenerate to Prym varieties, as was shown in \cite[Theorem 5]{alexeev15the-uniformization}. Therefore the codifferential of the map $PT_5$ is generically injective. The rest of this Section is devoted to the proof of the above result.

\begin{theorem}\label{thm:diff-bound}
Assume $[\pi_1\colon C_1 \rra R_1, p_1+ \cdots+ p_{22}]\in \hur_1$. If the map $PT$ is ramified at the point $$[C=C_1 \cup C_2 \rra R_1 \cup R_2]\in \ohur,$$ then, either $h^0 \bigl(C_1, \omega_{C_1}-2L_1\bigr)> 0$, or, the Prym-Tyurin canonical image of $C_1$ is contained in a quadric, in which case $h^0 (C_1, L_1) \geq 4$ and $h^0 (\oC, L) \geq 3$. Generically on $D_0$, none of these cases occur.
\end{theorem}

In what follows, we first recall the interpretation of the cotangent spaces to $\ocA_6$, $\ocA_{46}$, $\ocM_{46}$ and $\ohur$, then we describe the codifferential of $PT$.

\subsection{} Let $\oP$ be the usual compactification of the semi-abelian variety $PT(C,D)$ obtained by first completing $PT(C,D)$ to a $\bP^1$-bundle over the $5$-dimensional ppav $B := PT(C_1, D_1)$, and then identifying the $0$ and $\infty$-sections after translating by the extension datum of $PT(C,D)$ over $B$. We refer to \cite{mumford1983on-the-kodaira-dimension} for details.
The local to global spectral sequence induces the exact sequence

\[
\xymatrix{0 \ar[r] & H^0 (\cE xt^1_{\oP} (\Omega^1_{\oP}, \cO_{\oP}))^{\vee} \ar[r] & \Omega^1_{\ocA_6,[PT(C, D)]} \ar[r] & \Omega^1_{D_6,[PT(C,D)]} \ar[r] & 0,}
\]

where $\Omega^1_{D_6,[PT(C,D)]}$ is the cotangent space  to the boundary divisor $D_6$ of $\ocA_6$. Note that $\Omega^1_{D_6,[PT(C,D)]}$ is the dual to the space of deformations of $PT(C,D)$ that stay singular.
Let $\Omega^1_{\ocA_6} (\mbox{log} D_6)$ be the sheaf of $1$-forms with at worst simple logarithmic poles along $D_6$. By \cite[IV Proposition 3.1(vi), p. 107]{faltings1990degenerations-of-abelian}, the fiber $\Omega^1_{\ocA_6} (\mbox{log} D_6)_{[PT(C,D)]}$ can be identified with
$\Sym^2 H^0(C,\omega_C)^{(-5)}$, and this induces an identification
\[
\Omega^1_{D_6,[PT(C,D)]} = H^0 (C, \omega_C)^{(-5)} \odot H^0 (C_1, \omega_{C_1})^{(-5)},
\]
where
\[
H^0 (C, \omega_C)^{(-5)} \odot H^0 (C, \omega_{C_1})^{(-5)} :=\Bigl(H^0 (C, \omega_C)^{(-5)} \otimes H^0 (C_1, \omega_{C_1})^{(-5)}\Bigr) \bigcap \Sym^2 H^0(C,\omega_C)^{(-5)}.
\]
Remark that in this description $H^0(C_1, \omega_{C_1})^{(-5)}\subseteq H^0(C, \omega_C)^{(-5)}$ is a codimension one subspace.

\subsection{} The cotangent space to $\ocM_{46}$ at $[\oC]$ is $H^0 (\oC, \Omega^1_{\oC} \otimes \omega_{\oC})$. We have the natural map
\[
\Omega^1_{\oC} \otimes \omega_{\oC} \lra \rho_* (\Omega^1_C \otimes \omega_C),
\]
obtained from $\rho^* (\Omega^1_{\oC} \otimes \omega_{\oC}) \rra \Omega^1_C \otimes \omega_C$, which induces the map
\begin{equation}\label{eqCCbar}
H^0 (\oC, \Omega^1_{\oC} \otimes \omega_{\oC}) \lra H^0 (C, \Omega^1_C \otimes \omega_C).
\end{equation}
A local computation shows that the natural map $\omega_{\oC} \rra \rho_* \omega_C$ is an isomorphism. Therefore it induces an isomorphism
$
H^0 (\oC, \omega_{\oC}) \stackrel{\cong}{\lra} H^0 (C, \omega_C),
$
which shows that $H^0 (\oC, \omega_{\oC})$ is endowed with an endomorphism, which we still denote by  $D$, that is induced by the Kanev correspondence.

\subsection{} Let $\oJC$ denote the compactification of the Jacobian of $\oC$ described as the scheme parametrizing torsion-free sheaves of degree $0$ on $\oC$, see \cite{OdaSeshadri}. As above, we have the exact sequence
\[
\xymatrix{0 \ar[r] & H^0 \bigl(\cE xt^1_{\oJC} (\Omega^1_{\oJC}, \cO_{\oJC})\bigr)^{\vee} \ar[r] & \Omega^1_{\ocA_{46},[\oJC]} \ar[r] & H^0 (\oC,\omega_{\oC}) \odot H^0 (C_1,\omega_{C_1}) \ar[r] & 0,}
\]
where, again by \cite[IV Proposition 3.1(vi), p. 107]{faltings1990degenerations-of-abelian}, the space on the right classifies deformations of $\oJC$ of toric rank $6$. Here $H^0(C_1, \omega_{C_1})\subseteq H^0(\oC, \omega_{\oC})$ is viewed as a subspace of codimension $6$.

\subsection{} Consider the pull-back diagram
\[
\xymatrix{\overline{\cH}\ar^{\mathfrak{b}}[r]\ar^q[d]&\ocM_{0,24}\ar^p[d]\\
\overline{\Hur}\ar^{\mathfrak{br}}[r]&\tcM_{0,24}.}
\]
The ramification divisor of $p$ is the divisor $B_2$, its ramification index  being equal to $2$. The ramification divisor of $q$ is the divisor $E_0 + E_{\mathrm{azy}}$ \cite[Paragraph 6.11]{alexeev15the-uniformization}. Furthermore, $\mathfrak{b}^*(B_2)=E_0+3E_{\mathrm{azy}}+2E_{\mathrm{syz}}$. It follows that the map $\mathfrak{br}$ is generically unramified along $D_0$ and we can identify the  cotangent space $\Omega^1_{\ohur , [C,\pi]}$ with $H^0(R, \Omega^1_R \otimes \omega_R(B))$ which is the cotangent space to $\tcM_{0,24}$.

\begin{definition}
Let $M$ and $A$ be the ramification and anti-ramification
divisors of the $\WE$-admissible cover $\pi\colon C\rightarrow R$. As $M$ and $A$ are supported on the smooth locus of $C$, we have the usual identities
\begin{equation}\label{eqramif}
    \pi^*(B) = 2M + A, \quad
    \Omega^1_C = \pi^*(\Omega^1_R) (M), \quad  \omega_C = \pi^*(\omega_R) (M), \quad
    \Omega^1_C \otimes \omega_C(A) = \pi^*(\Omega^1_R \otimes \omega_R(B)),
  \end{equation}
\end{definition}
\noindent and we can define the trace map as for smooth covers:
\begin{definition}
Let $\mbox{tr}\colon \pi_* \cO_C(-A) \to \cO_R$ be the trace map on regular
functions. For an open affine subset $U\subseteq \bP^1$, a regular
function $\varphi\in \Gamma(U, \cO_C(-A))$, and  a point $y\in U$, one has
\[
\mbox{tr}(\varphi)(y) = \sum_{x\in f\inv(y)} \varphi(x),
\]
counted with
multiplicities. Note that $\mbox{tr}$ is surjective.
By \eqref{eqramif}, the trace map induces the map $\pi_* (\Omega^1_C \otimes \omega_C) \to \Omega^1_R \otimes \omega_R(B)$. Let $\mbox{Tr}\colon H^0(C, \Omega^1_C \otimes \omega_C) \to
H^0\bigl(R, \Omega^1_R \otimes \omega_R(B)\bigr)$ be the induced map on global sections. The composition of Tr with the map \eqref{eqCCbar}
\[
\overline{\mbox{Tr}} \colon H^0 (\oC, \Omega^1_{\oC} \otimes \omega_{\oC}) \lra H^0 (C, \Omega^1_C \otimes \omega_C) \lra H^0(R, \Omega^1_R \otimes \omega_R(B))
\]
can be viewed as the codifferential of the forgetful map $\overline{\Hur} \rra \ocM_{46}$ at the point $[C, \pi]$.
\end{definition}

\begin{proposition}\label{propcodiff}
The codifferential $(dPT)^{\vee}_{[C,\pi]} \colon T^{\vee}_{[PT(C,D)]}(\overline{\cA}_6) \rightarrow T^{\vee}_{[C,\pi]}\bigl(\ohur \bigr)$ is given by the following
composition of maps:
\begin{equation}\label{eqcodiff}
T^{\vee}_{[PT(C,D)]}\bigl(\ocA_6\bigr)
\into T^{\vee}_{[\oJC]}\bigl(\ocA_{46}\bigr)
\xrightarrow{\mathrm{tor}} H^0(\oC, \Omega^1_{\oC} \otimes \omega_{\oC})
\stackrel{\overline{\rm Tr}}{\lra} H^0\bigl(R, \Omega^1_R \otimes \omega_R(B)\bigr),
\end{equation}
where the second map is the codifferential of the Torelli
map $\ocM_{46}\to \ocA_{46}$.
\end{proposition}
\begin{proof}
Follows along the lines of the proof of \cite[Theorem 10.3]{alexeev15the-uniformization} (which treats the same question in the case of a point $[C, \pi]\in \hur$ corresponding to a smooth source curve)  with obvious modifications. The first map in (\ref{eqcodiff}) is the codifferential of the map from the perfect cone compactification of the moduli space of ppav of dimension $46$ having an endomorphism $D$ with eigenvalues $+1$ and $-5$ of eigenspaces of dimensions $40$ and $6$ respectively to $\overline{\cA}_6$.
\end{proof}

\subsection{} We first study the codifferential $dPT^{\vee}$ on the conormal space to the boundary divisor $D_6$ of $\overline{\cA}_6$. To that end,  we first describe locally differentials on $C, \oC$ and $R$ near the node $q$ of $R$ corresponding to the point described in (\ref{bdpoint}).

Choose local coordinates $t$ on $R_1$ and $s$ on $R_2$ at the node $q$ of $R$. These can be identified via $\pi$ with local coordinates at the nodes $o_1, \ldots , o_{27}$ of $C$ above $q$. Then the stalks of the sheaves $\Omega^1_R, \omega_R, \Omega^1_C, \omega_C, \Omega^1_R\otimes \omega_R, \Omega^1_C\otimes \omega_C$ at their nodes have the following presentations
\[
\begin{array}{ll}
\Omega^1_{R,q}, \Omega^1_{C, o_i} : & \cO \langle ds, dt \rangle / \left( t ds + s dt \right) \\
\omega_{R,q}, \omega_{C, o_i} : & \cO \big \langle \frac{ds}{s}, \frac{dt}{t} \big \rangle \big / \left( \frac{ds}{s} + \frac{dt}{t} \right) \\
\Omega^1_{R,q} \otimes \omega_{R,q}, \Omega^1_{C, o_i}\otimes \omega_{C, o_i} : & \cO \Big \langle \frac{(ds)^2}{s}, \frac{(dt)^2}{t} \Big \rangle \Big / \left( t \frac{(ds)^2}{s} - s \frac{(dt)^2}{t} \right).
\end{array}
\]
We have the natural exact sequence on $R$
\[
0 \lra \mathrm{Tors}(\Omega_R^1) \lra \Omega^1_R \stackrel{\iota_R}{\lra} \omega_R \lra \bC_q \lra 0
\]
where $\mathrm{Tors}(\Omega_R^1)\cong \mathbb C_q$ is a sky-scraper sheaf at $q$ generated by the torsion differential $s dt = -t ds$. From this, by tensoring with the locally free sheaf $\omega_R$  we obtain the exact sequence
\[
0 \lra \bC_q \lra \Omega^1_R \otimes \omega_R \stackrel{\kappa_R}{\lra} \omega^{\otimes 2}_R \lra \bC_q \lra 0
\]
where the kernel of $\kappa_R$ is generated by $ds\, dt = s \frac{(dt)^2}{t} = t \frac{(ds)^2}{s}$. One has a similar exact sequence  for $C$ at the points $o_i$.
A torsion section $\gamma \in H^0(C, \Omega^1_C\otimes \omega_C)$ can be written as
\[
\gamma = \lambda_i t\frac{(ds)^2}{s} = \lambda_i s \frac{(dt)^2}{t} \ \
  \text{near } o_i\in C.
\]

\subsection{Local description at the nodes}\label{Parcoordinates} Assume the nodes $o_1, \ldots, o_{27}$ of $C$ are labeled in such a way that $o_{2i-1}$ and $o_{2i}$ map to the node $u_i$ of $\oC$ for $i=1, \ldots, 6$. Labeling by $s_i, t_i$ the local coordinates on the two branches of $C_2$ and $C_1$ at the point $o_i$ for $i=1, \ldots, 27$, then $t_{2i-1}, t_{2i}$ are local coordinates at the point $u_i\in \oC$ for $i=1, \ldots, 6$. We have the natural commutative diagram of exact sequences
\[
\xymatrix{0\ar[r] & \bigoplus_{i=1}^6 \bC_{u_i} \ar[r]\ar^{\rho^*}[d]& \Omega^1_{\oC}\otimes \omega_{\oC}\ar^{\kappa_{\oC}}[r]\ar^{\rho^*}[d]&\omega_{\oC}^{\otimes 2}\ar^{\rho^*}[d]\\
0\ar[r] & \rho_* \left(\bigoplus_{i=1}^{27} \bC_{o_i}\right) \ar[r] & \rho_*\left(\Omega^1_{C}\otimes \omega_{C}\right)\ar^{\rho_*\kappa_C}[r]&\rho_*\omega_C^{\otimes 2}}
\]
where $\bigoplus_{i=1}^6\mathbb C_{u_i}$ is the torsion subsheaf of $\Omega_{\oC}^1\otimes \omega_{\oC}$. The torsion part $\bigoplus_{i=1}^{27} \bC_{o_i}$ of $\Omega_C^1\otimes \omega_C$ has an action of the correspondence $D$ which leaves the image of $\bigoplus_{i=1}^6 \bC_{u_i}$  invariant. The action of $D$ on this subspace has two eigenspaces of dimensions $1$ and $5$ for the eigenvalues $-5$ and $+1$ respectively. The proof of this is analogous to \cite[Lemma 10.8]{alexeev15the-uniformization}.

A torsion section $\ogamma$ of $\Omega^1_{\oC}\otimes \omega_{\oC}$ can be locally written near $u_i\in \oC$ as
\[
\ogamma = \mu_i t_{2i}\frac{(dt_{2i-1})^2}{t_{2i-1}} = \mu_i t_{2i-1} \frac{(dt_{2i})^2}{t_{2i}}, \ \
  \text{ where }\  \mu_i\in \mathbb C.
\]
Identifying the local coordinates on $C$ with those on $R$ as in the previous paragraph, a generator of the $(-5)$-eigenspace is the section $\ogamma\in \mathrm{Tors}(\Omega_{\oC}^1\otimes \omega_{\oC})$ with $\mu_i =1$ for $i=1, \ldots , 6$.

\subsection{Injectivity in conormal directions} By Proposition \ref{propcodiff}, the map $PT$ is ramified at $[C,\pi]\in\ohur$ if the kernel of the composition of maps (\ref{eqcodiff}) is nonzero. Each of the above cotangent spaces has a natural subspace which is the conormal space to the equisingular deformations. Restricting the above sequence to each conormal space appearing in (\ref{eqcodiff}), we obtain the exact sequence:
\begin{equation}\label{eqtorsion}
H^0 \bigl(\cE xt^1_{\oP} (\Omega^1_{\oP}, \cO_{\oP})\bigr)^{\vee}
\into H^0 \bigl(\cE xt^1_{\oJC} (\Omega^1_{\oJC}, \cO_{\oJC})\bigr)^{\vee}
\xrightarrow{\mathrm{tor}} H^0 \bigl(\cE xt^1_{\oC} (\Omega^1_{\oC}, \cO_{\oC})\bigr)^{\vee}
\stackrel{\overline{\rm Tr}}{\lra} H^0 \bigl(\cE xt^1_{R} (\Omega^1_{R}(B), \cO_{R})\bigr)^{\vee}
\end{equation}

Using  e.g., \cite[Corollary 15.4]{Andreatta}, the map $\mbox{tor}$ in \eqref{eqtorsion} is an isomorphism. Identifying the second and third space in \eqref{eqtorsion}, by Paragraph \ref{Parcoordinates}, the second space has an action of the correspondence $D$ and the image of the first arrow is the $1$-dimensional eigenspace for the eigenvalue $-5$. With our earlier choice of bases (see \ref{Parcoordinates}), a generator of the $(-5)$-eigenspace is the element $\sum_{i=1}^6 t_{2i}\frac{(dt_{2i-1})^2}{t_{2i-1}}$. The image of an element $\sum_{i=1}^6 \mu_i t_{2i}\frac{(dt_{2i-1})^2}{t_{2i-1}}$ in the last space is $\sum_{i=1}^6 \mu_i t\frac{(ds)^2}{s}$. It follows that the composition above is an isomorphism between two $1$-dimensional spaces.

\vskip 3pt

Note that, via push-forward to $R_1$, we have the following identification
\[
H^0(R, \Omega^1_R \otimes \omega_R(B)) \cong \mbox{Tors}_q \bigl(\Omega_R\otimes \omega_R(B)\bigr)\oplus H^0\bigl(R_1,\omega^{\otimes 2}_{R_1}(B_1+q)\bigr) \cong \mathbb C_q\oplus H^0\bigl(R_1, \omega^{\otimes 2}_{R_1}(B_1+q)\bigr),
\]
where $B_1=p_1+\cdots+p_{22}$ and the skyscraper sheaf $\mathbb C_q$ is generated by $ds\, dt= s \frac{(dt)^2}{t} = t \frac{(ds)^2}{s}$. The image of $H^0 (\oC, \Omega^1_{\oC}\otimes \omega_{\oC})$ in $H^0 (\oC, \omega_{\oC}^{\otimes 2})$ is the space of sections vanishing at the nodes of $\oC$. This image will be then identified with $H^0 \bigl(C_1, \omega_{C_1}^{\otimes 2} (o_1 + \cdots + o_{12})\bigr) \subseteq H^0 (\oC, \omega_{\oC}^{\otimes 2}) \subseteq H^0 \bigl(C_1, \omega_{C_1}^{\otimes 2} (2o_1 + \cdots + 2o_{12})\bigr)$.

\subsection{} \label{pardiagram} Taking the quotient of the exact sequence \eqref{eqcodiff} by \eqref{eqtorsion}, we obtain the commutative diagram
\[
\xymatrix{T^{\vee}_{[PT(C,D)]}(\overline{\cA}_6) \ar[r] \ar[d] & T^{\vee}_{[\oJC]}(\overline{\cA}_{46}) \ar[d] \ar[r] & \\
H^0 (\oC,\omega_{\oC})^{(-5)} \odot H^0 (C_1,\omega_{C_1})^{(-5)} \ar[r] & H^0 (\oC,\omega_{\oC}) \odot H^0 (C_1,\omega_{C_1}) \ar[r] &}
\]
\[
\xymatrix{ \ar[r] & H^0 (\oC, \Omega^1_{\oC}\otimes \omega_{\oC})\ar^{\overline{\rm Tr}}[r]\ar^{}[d]&H^0(R, \Omega^1_R \otimes \omega_R(B)) \ar^{}[d]\\
\ar[r] & H^0 \bigl(C_1, \omega_{C_1}^{\otimes 2} (o_1 + \cdots + o_{12})\bigr) \ar^{\overline{\rm tr}}[r]&H^0\bigl(R_1, \omega^{\otimes 2}_{R_1}(B_1+q)\bigr).}
\]

To summarize the discussion above,  the injectivity of the codifferential of $PT$ at the point $[C, \pi]\in D_0$ is equivalent to the injectivity of the composition in the bottom row above.

\subsection{The kernel of $\overline{\mbox{tr}}$} For each of the
branch points $p_i\in R_1$ with $i=1, \ldots, 22$, let $\{r_{ij}\}_{j=1}^6\subseteq C_1$ be the
ramification points lying over $p_i$. The formal neighborhoods of the
points $r_{ij}$ are naturally identified, so that we can choose a
single local parameter $x$ and write a section
$\gamma\in H^0\bigl(C_1, \omega_{C_1}^{\otimes 2} (o_1 + \cdots + o_{12})\bigr)$ as
\begin{displaymath}
  \gamma = \varphi_{ij}(x) \cdot (dx)^{2} \ \
  \text{near } r_{ij}\in C.
\end{displaymath}
Choose a local parameter $y$ at the point $p_i$, so that $\pi |_{C_1}$ is given locally by the map $y=x^2$. We
can use the same local parameter at the remaining $15$ antiramification points $\{q_{ik}\}_{k=1}^{15}$
over $p_i$ at which $\pi$ is unramified, and write $\gamma =
\psi_{ik}(y)\cdot (dy)^{2}$ near $q_{ik}\in C$, for $k=1, \ldots, 15$.

At the point $q$, we similarly choose a local parameter $x$ and identify it with the local parameters at the points $o_1, \ldots, o_{27}$. Write $\gamma = \rho_i (x) \frac{(dx)^{2}}{x}$ near $o_i$ for $i=1, \ldots , 12$.

\begin{lemma}\label{lem:ker-of-trace}
  The kernel of the trace map $\overline{\mathrm{tr}}\colon H^0\bigl(C_1, \omega_{C_1}^{\otimes 2} (o_1 + \cdots + o_{12})\bigr) \to H^0\bigl(R_1, \omega^{\otimes 2}_{R_1}(B_1+q)\bigr)$ consists of those quadratic differentials $\gamma$ which, using the previous notation, satisfy
\[
\sum_{j=1}^6 \varphi_{ij}(r_{ij}) =0, \ \ \text{for} \ \ i=1, \ldots, 22, \quad  \text{and} \quad \sum_{j=1}^{12} \rho_j (o_j)= 0.
\]
\end{lemma}
\begin{proof}
Local calculation, very similar to the proof of \cite[Lemma 10.5]{alexeev15the-uniformization}.
\end{proof}

\vskip 3pt

We are now in a position to describe set-theoretically the ramification of the map $PT\colon \hur\rightarrow \mathcal{A}_6$ along $D_0$, which then quickly leads to an alternative proof of the dominance of $PT$.

\vskip 6pt

\noindent \emph{Proof of Theorem \ref{thm:diff-bound}.} The global sections of $\omega_{\oC}$ can be identified with the sections of $\omega_{C_1}(o_1+\cdots +o_{12})$ whose residues at $o_{2i-1}$ and $o_{2i}$ are opposite for $i=1, \ldots, 6$. A proof analogous to that of \cite[Lemma 10.8]{alexeev15the-uniformization} shows that, under this identification, the elements of $H^0 (\oC,\omega_{\oC})^{(-5)}$ correspond to sections having \emph{the same residue} at  $o_{2i-1}$ and $o_{2i}$ for $i=1, \ldots, 6$ (in addition to opposite residues at $o_{2i-1}$ and $o_{2i}$). This first implies that the points $o_1, \ldots , o_{12}$ have the same image, say $\overline{o}$, in the Prym-Tyurin canonical space $\bP \bigl(H^0 (\oC,\omega_{\oC})^{(-5)}\bigr)^{\vee}\cong \bP^5$. Next, using Lemma \ref{lem:ker-of-trace}, we deduce that if an element $$\beta\in H^0 (\oC,\omega_{\oC})^{(-5)} \odot H^0 (C_1,\omega_{C_1})^{(-5)}$$ belongs to the kernel of the composition on the bottom row of the diagram in paragraph \ref{pardiagram}, then its image in $H^0\bigl(C_1, \omega_{C_1}^2 (o_1 + \cdots + o_{12})\bigl)$ belongs to the subspace
$$H^0 \Bigl(C_1, \omega_{C_1}^{\otimes 2} \bigl(-\sum_{i=1}^{22}\sum_{j=1}^{6} r_{ij}\bigr)\Bigr) = H^0 (C_1 , \omega_{C_1} \otimes L_1^{\otimes (-2)}).$$

Assuming $H^0 (C_1 , \omega_{C_1} \otimes L_1^{\otimes (-2)})= 0$, and regarding $\beta$ as an element of $\mbox{Sym}^2 H^0 (\oC,\omega_{\oC})^{(-5)}$, we obtain that $\beta$ is the equation of a quadric containing the image of $\oC$ in the Prym-Tyurin canonical space $\bP \bigl(H^0 (\oC,\omega_{\oC})^{(-5)}\bigr)^{\vee}$.

\vskip 4pt

Since, as explained, $PT_1\colon \hur_{1}\rightarrow \mathcal{A}_5$ is dominant, we may assume via Theorem \ref{thm:hur1dom} that the Prym-Tyurin canonical image of $C_1$ in $\bP^4$ is not contained in a quadric. It follows that the quadric defined by $\beta$ is not a pull-back from
$\bP \bigl(H^0 (C_1,\omega_{C_1})^{(-5)}\bigr)^{\vee}$ via the projection from $\overline{o}$. Therefore this quadric is {\em not} singular at $\overline{o}$ and its tangent hyperplane at $\overline{o}$ contains the lines tangent to the Prym-Tyurin canonical image of $\oC$. The image of this tangent hyperplane in $\bP \bigl(H^0 (C_1,\omega_{C_1})^{(-5)}\bigr)^{\vee}$  contains the images of $o_1, \ldots, o_{12}$. In other words, the image of $H^0\bigl(\cO_{o_1+\cdots+o_{12}}(\Gamma))$ by the map $\alpha_1$ in the sequence \eqref{cohstandard} is contained in a hyperplane. This first implies that $h^0 (C_1, L_1 ) \geq 4$. Next, since the $(-5)$-eigenspace in $H^0 (\cO_{\Gamma}(\Gamma))$ can be identified with the primitive Picard group of a smooth cubic surface, having the same value at each pair of points $o_{2i-1}, o_{2i}$ for $i=1, \ldots, 6$ imposes \emph{only one} condition on the sections of $L_1$. Hence we always have $h^0 (\oC, L) \geq h^0 (C_1, L_1) -1$, and, in this case, $h^0 (\oC, L) \geq 3$.

\vskip 3pt

The fact that these situations do not occur for a general choice of a point of $D_0$ is a consequence of Theorem \ref{thm:glued-curve-comps}, for the $W(E_6)$-admissible cover constructed there lies in $D_0$.
\qed

\begin{corollary}
  The Prym-Tyurin map $PT\colon \ohur\dashrightarrow \mathcal{A}_6$ is generically
  finite.
\end{corollary}
\begin{proof}
Indeed, the above shows that the differential of $PT$ on tangent
spaces is generically an isomorphism.
\end{proof}

\renewcommand{\MR}[1]{}
\bibliographystyle{amsalpha}

\newcommand{\etalchar}[1]{$^{#1}$}
\def\cprime{$'$}
\providecommand{\bysame}{\leavevmode\hbox to3em{\hrulefill}\thinspace}
\providecommand{\MR}{\relax\ifhmode\unskip\space\fi MR }
\providecommand{\MRhref}[2]{%
  \href{https://urldefense.com/v3/__http://www.ams.org/mathscinet-getitem?mr=*1*7D*7B*2__;IyUlIw!!Mih3wA!V-MMkuVgwimme2TQl3kbJi3SOoj3gtnii2TRyowa5fU-igauPK0orufJxezjf3E$ }
}
\providecommand{\href}[2]{#2}

\vskip 16pt

\section*{Appendix: the character table of $\WE$}

\vskip 8pt

At several points in this paper we have used the character table of $W(E_6)$. We record it in the form presented by GAP \cite{GAP4} by
applying the command {\tt Display(CharacterTable("W(E6)"))}. It is
also the same as the table in Atlas
\cite[p.27]{conway1985atlas-finite} for the group $U_4(2).2=\WE$,
obtained from the character table of $U_4(2)$ by the splitting and fusion
rules.  As usual, rows are for characters (we added convenient
names in column 2), and columns are for conjugacy classes.

\vskip 18pt

\begin{table}[htp!]\centering
\setlength{\tabcolsep}{1pt}
\begin{tabular}{|rr|rrrrrrrrrrrrrrr|rrrrrrrrrr|} \hline
$\chi$ & name & 1a & 2a & 2b & 3a & 3b & 3c & 4a & 4b & 5a & 6a & 6b & 6c & 6d & 9a & 12a & 2c & 2d & 4c & 4d & 6e & 6f & 6g & 8a & 10a & 12b \\
\hline
1 & 1\ {} & 1 & 1 & 1 & 1 & 1 & 1 & 1 & 1 & 1 & 1 & 1 & 1 & 1 & 1 & 1 & 1 & 1 & 1 & 1 & 1 & 1 & 1 & 1 & 1 & 1 \\
2 & \tov{1}\ {} & 1 & 1 & 1 & 1 & 1 & 1 & 1 & 1 & 1 & 1 & 1 & 1 & 1 & 1 & 1 & -1 & -1 & -1 & -1 & -1 & -1 & -1 & -1 & -1 & -1 \\
3 & 10\ {} & 10 & -6 & 2 & 1 & -2 & 4 & 2 & -2 & . & -3 & . & . & 2 & 1 & -1 & . & . & . & . & . & . & . & . & . & . \\
4 & 6\ {} & 6 & -2 & 2 & -3 & 3 & . & 2 & . & 1 & 1 & 1 & -2 & -1 & . & -1 & 4 & . & -2 & 2 & 1 & -2 & . & . & -1 & 1 \\
5 & \tov{6}\ {} & 6 & -2 & 2 & -3 & 3 & . & 2 & . & 1 & 1 & 1 & -2 & -1 & . & -1 & -4 & . & 2 & -2 & -1 & 2 & . & . & 1 & -1 \\
6 & 20a\ {} & 20 & 4 & -4 & -7 & 2 & 2 & 4 & . & . & 1 & -2 & -2 & 2 & -1 & 1 & . & . & . & . & . & . & . & . & . & . \\
7 & 15a\ {} & 15 & -1 & -1 & 6 & 3 & . & 3 & -1 & . & 2 & -1 & 2 & -1 & . & . & 5 & -3 & 1 & 1 & -1 & 2 & . & -1 & . & 1 \\
8 & \tov{15a}\ {} & 15 & -1 & -1 & 6 & 3 & . & 3 & -1 & . & 2 & -1 & 2 & -1 & . & . & -5 & 3 & -1 & -1 & 1 & -2 & . & 1 & . & -1 \\
9 & 15b\ {} & 15 & 7 & 3 & -3 & . & 3 & -1 & 1 & . & 1 & -2 & 1 & . & . & -1 & 5 & 1 & 3 & -1 & 2 & -1 & 1 & -1 & . & . \\
10 & \tov{15b}\ {} & 15 & 7 & 3 & -3 & . & 3 & -1 & 1 & . & 1 & -2 & 1 & . & . & -1 & -5 & -1 & -3 & 1 & -2 & 1 & -1 & 1 & . & . \\
11 & 20b\ {} & 20 & 4 & 4 & 2 & 5 & -1 & . & . & . & -2 & 1 & 1 & 1 & -1 & . & 10 & 2 & 2 & 2 & 1 & 1 & -1 & . & . & -1 \\
12 & \tov{20b}\ {} & 20 & 4 & 4 & 2 & 5 & -1 & . & . & . & -2 & 1 & 1 & 1 & -1 & . & -10 & -2 & -2 & -2 & -1 & -1 & 1 & . & . & 1 \\
13 & 24\ {} & 24 & 8 & . & 6 & . & 3 & . & . & -1 & 2 & 2 & -1 & . & . & . & 4 & 4 & . & . & -2 & 1 & 1 & . & -1 & . \\
14 & \tov{24}\ {} & 24 & 8 & . & 6 & . & 3 & . & . & -1 & 2 & 2 & -1 & . & . & . & -4 & -4 & . & . & 2 & -1 & -1 & . & 1 & . \\
15 & 30\ {} & 30 & -10 & 2 & 3 & 3 & 3 & -2 & . & . & -1 & -1 & -1 & -1 & . & 1 & 10 & -2 & -4 & . & 1 & 1 & 1 & . & . & -1 \\
16 & \tov{30}\ {} & 30 & -10 & 2 & 3 & 3 & 3 & -2 & . & . & -1 & -1 & -1 & -1 & . & 1 & -10 & 2 & 4 & . & -1 & -1 & -1 & . & . & 1 \\
17 & 60a\ {} & 60 & 12 & 4 & -3 & -6 & . & 4 & . & . & -3 & . & . & -2 & . & 1 & . & . & . & . & . & . & . & . & . & . \\
18 & 80\ {} & 80 & -16 & . & -10 & -4 & 2 & . & . & . & 2 & 2 & 2 & . & -1 & . & . & . & . & . & . & . & . & . & . & . \\
19 & 90\ {} & 90 & -6 & -6 & 9 & . & . & 2 & 2 & . & -3 & . & . & . & . & -1 & . & . & . & . & . & . & . & . & . & . \\
20 & 60b\ {} & 60 & -4 & 4 & 6 & -3 & -3 & . & . & . & 2 & -1 & -1 & 1 & . & . & 10 & 2 & -2 & -2 & 1 & 1 & -1 & . & . & 1 \\
21 & \tov{60b}\ {} & 60 & -4 & 4 & 6 & -3 & -3 & . & . & . & 2 & -1 & -1 & 1 & . & . & -10 & -2 & 2 & 2 & -1 & -1 & 1 & . & . & -1 \\
22 & 64\ {} & 64 & . & . & -8 & 4 & -2 & . & . & -1 & . & . & . & . & 1 & . & 16 & . & . & . & -2 & -2 & . & . & 1 & . \\
23 & \tov{64}\ {} & 64 & . & . & -8 & 4 & -2 & . & . & -1 & . & . & . & . & 1 & . & -16 & . & . & . & 2 & 2 & . & . & -1 & . \\
24 & 81\ {} & 81 & 9 & -3 & . & . & . & -3 & -1 & 1 & . & . & . & . & . & . & 9 & -3 & 3 & -1 & . & . & . & 1 & -1 & . \\
25 & \tov{81}\ {} & 81 & 9 & -3 & . & . & . & -3 & -1 & 1 & . & . & . & . & . & . & -9 & 3 & -3 & 1 & . & . & . & -1 & 1 & . \\
\hline \end{tabular}
\medskip\caption{The character table of $W(E_6)$}
\label{tab:E6-chartable}
\end{table}

\end{document}